\documentclass[]{amsart}

\usepackage{geometry}
\usepackage[utf8]{inputenc}
\usepackage[T1]{fontenc}

\usepackage{amsmath}
\usepackage{amsthm}
\usepackage{enumitem}
\usepackage{amsfonts}
\usepackage{mathtools}
\usepackage{mathrsfs} 
\usepackage{courier} 
\usepackage{upgreek} 
\usepackage{dsfont}
\usepackage{hyperref} 
\usepackage[table,xcdraw]{xcolor}

\definecolor{applegreen}{rgb}{0.55, 0.71, 0.0}

\usepackage{caption}
		\newcommand{\RNum}[1]{\uppercase\expandafter{\romannumeral #1\relax}}

\newcommand{\peter}{\mathcal{P}}

\DeclarePairedDelimiter\set{\{}{\}}
\DeclarePairedDelimiterX\lrangle[2]{\langle}{\rangle}{#1, #2}

\usepackage{subcaption}

\usepackage{multirow}

\newtheorem{theorem}{Theorem}
\newtheorem{lemma}[theorem]{Lemma}
\newtheorem{hypo}[theorem]{Hypothesis}
\newtheorem{corollary}[theorem]{Corollary}
\newtheorem{proposition}[theorem]{Proposition}
\newtheorem{conjecture}[theorem]{Conjecture}

\newtheorem{observation}[theorem]{Observation}

\numberwithin{theorem}{section}

\theoremstyle{definition}
\newtheorem{definition}[theorem]{Definition}

\theoremstyle{remark}

\makeatletter
\@namedef{subjclassname@2020}{%
	\textup{2020} Mathematics Subject Classification}
\makeatother

\usepackage{tikz}
\usetikzlibrary{arrows}

\begin{document}
\title{4-cop-win graphs have at least 19 vertices}

\author{J\'er\'emie Turcotte}
\address{D\'{e}partment de math\'{e}matiques et de statistique, Universit\'{e}
de Montr\'{e}al, Montr\'eal, Canada}
\email{mail@jeremieturcotte.com}
\urladdr{www.jeremieturcotte.com}

\author{Samuel Yvon}
\address{D\'{e}partement d'informatique et de recherche op\'{e}rationnelle,
Universit\'{e} de Montr\'{e}al, Montr\'{e}al, Canada}
\email{samuel.yvon@umontreal.ca}
\urladdr{}


\maketitle
\noindent

\begin{abstract}
	We show that the cop number of any graph on 18 or fewer vertices is at most
	3. This answers a question posed by Andreae in 1986, as well as more
	recently by Baird et al. We also find all 3-cop-win graphs on 11 vertices,
	narrow down the possible 4-cop-win graphs on 19 vertices and make some
	progress on finding the minimum order of 3-cop-win planar graphs.
\end{abstract}


\section{Introduction}
The game of cops and robbers was first defined by Quilliot in
\cite{quilliot_problemes_1978} and Nowakowski and Winkler in
\cite{nowakowski_vertex--vertex_1983}. Playing on the vertices of a connected,
undirected and finite graph, a group of cops pursues a robber. On the first turn
of the game, each cop selects a vertex as its initial position, followed by the
robber. The cops and the robber then alternate turns with the cops going first. On the
cops' turn (resp. robber's turn), each cop (resp. the robber) must choose either
to remain on its current vertex or go to a vertex in the neighbourhood of its
current position. If at some point during the game a cop and the robber are on
the same vertex, the robber is said to be caught and the cops win. The robber
wins if it has a strategy ensuring it is never caught by the cops. At all times,
the positions of the cops and of the robber are known by all. Furthermore, the
cops may coordinate their strategies, and are allowed to share vertices. An
alternative formulation is to define the game of cops and robbers as a
two-player game in which the first player controls some number of pieces named
the cops, the second player controls a single piece named the robber, and the
graph is the game board (for example as in \cite{joret_cops_2008}).

For a connected graph $G$, we denote by $c(G)$ the minimum number of cops which
can always catch the robber on $G$. Introduced by Aigner and Fromme in
\cite{aigner_game_1984}, $c(G)$ is called the \emph{cop number} of $G$. If
$c(G)=k$, we say $G$ is $k$-cop-win.

The cop number has been the main focus of most articles on cops and robbers, but
other parameters, such as the capture time, are also studied. See
\cite{bonato_conjectures_2016} for a quick overview of this field or
\cite{bonato_game_2011} for a more in-depth introduction. Although in this paper
we study the classical version of the game, a multitude of variants of this game
have been considered in recent years. For instance, we can allow the robber to
move multiple edges at once \cite{frieze_variations_2011}, or loosen the winning
condition to the cops being within some distance of the robber
\cite{bonato_cops_2010}.

While there has been a significant amount of research on the cop number, often
on specific classes of graphs (for instance in
\cite{andreae_pursuit_1986,frankl_pursuit_1987,joret_cops_2008}), there are
still surprisingly many elementary open questions. We consider here the problem
of finding the minimum order of a $k$-cop-win graph, more specifically for
$k=4$. This question was first posed by Andreae in \cite{andreae_pursuit_1986}.
It is also raised, seemingly independently, by Baird et al.~in
\cite{baird_minimum_2014}.

The case of $k=3$ has already been solved. Andreae claimed without proof in
\cite{andreae_pursuit_1986} that the Petersen graph (see Figure
\ref{fig:petersenrobertson_graph}) is the unique smallest 3-cop-win graph. This
statement was later proved by Baird et al.~in \cite{baird_minimum_2014}.

For a given graph $G$, we denote by $V(G)$ and $E(G)$, respectively, the set of
vertices and of edges of $G$. We denote by $M_k$ the minimum order of a
$k$-cop-win graph;  formally, $$M_k=\min\{|V(G)| : G \text{ is a connected graph, }
c(G)=k\}.$$Hosseini proved in \cite{hosseini_note_2018} that $M_{k}<M_{k+1}$,
confirming the intuition that if one scans all graphs by increasing order, one
cannot find a $(k+1)$-cop-win graph before finding a $k$-cop-win graph.

Although interesting by itself, the concept of minimum $k$-cop-win graphs is
also useful in regards to Meyniel's conjecture \cite{frankl_cops_1987}, which
asks whether the cop number $G$ is $O(\sqrt{n})$, where $n$ is the order of $G$.
Currently the best known upper bound is $n2^{-(1+o(1))\sqrt{\log_2 n}}$, as
proved by Lu and Peng \cite{lu_meyniels_2012}, Scott and Sudakov
\cite{scott_bound_2011}, and Frieze, Krivelevich and Loh
\cite{frieze_variations_2011}. A deeper survey of Meyniel's conjecture is
\cite{baird_meyniels_2013}. It was proved by Baird et al.~in
\cite{baird_minimum_2014} that Meyniel's conjecture is equivalent to proving
that $M_k\in \Omega(k^2)$.

The specific problem of finding the minimum order of a 4-cop-win graph has
received some interest. Hosseini proved in \cite{hosseini_game_2018} that
$M_4\geq 16$, and that such a minimal graph is 3-connected, provided it does not
contain a vertex of degree 2. The problem is also briefly referenced in
\cite{bonato_conjectures_2016}. It was suggested in
\cite{andreae_pursuit_1986,baird_minimum_2014} that the value of $M_4$ might be
19. Indeed, the smallest known 4-cop-win graph is the Robertson graph (see
Figure \ref{fig:petersenrobertson_graph}). This graph was first discovered by
Robertson in \cite{robertson_smallest_1964} as the smallest 4-regular graph with
girth 5.

Aigner and Fromme proved in \cite{aigner_game_1984} that graphs with girth at
least 5 have a cop number of a least their minimum degree. This result has since
been generalized by Frankl in \cite{frankl_cops_1987}, and more recently by
Bradshaw et al.~in \cite{bradshaw_cop_2020}, where they prove stronger lower
bounds on the cop number of graphs with high girth. One deduces that the cop
number of the Robertson graph is therefore at least 4. It is easily seen in
Figure \ref{fig:petersenrobertson_graph} that placing a cop on each of the three
exterior vertices ($a,b,c$) only leaves 4 unprotected vertices, which form
independent edges. A last cop may then capture the robber, hence the Robertson
graph is 4-cop-win (this argument appears in \cite{beveridge_petersen_2012}).

More generally, it was suggested in \cite{baird_minimum_2014} that the smallest
$d$-cop-win graphs might be the $(d,5)$-cages for $d\geq 3$. A $(d,g)$-cage is a
regular graph of degree $d$ and girth $g$ of minimum order. In particular, the
Petersen graph is the unique $(3,5)$-cage and the Robertson graph is the unique
$(4,5)$-cage.

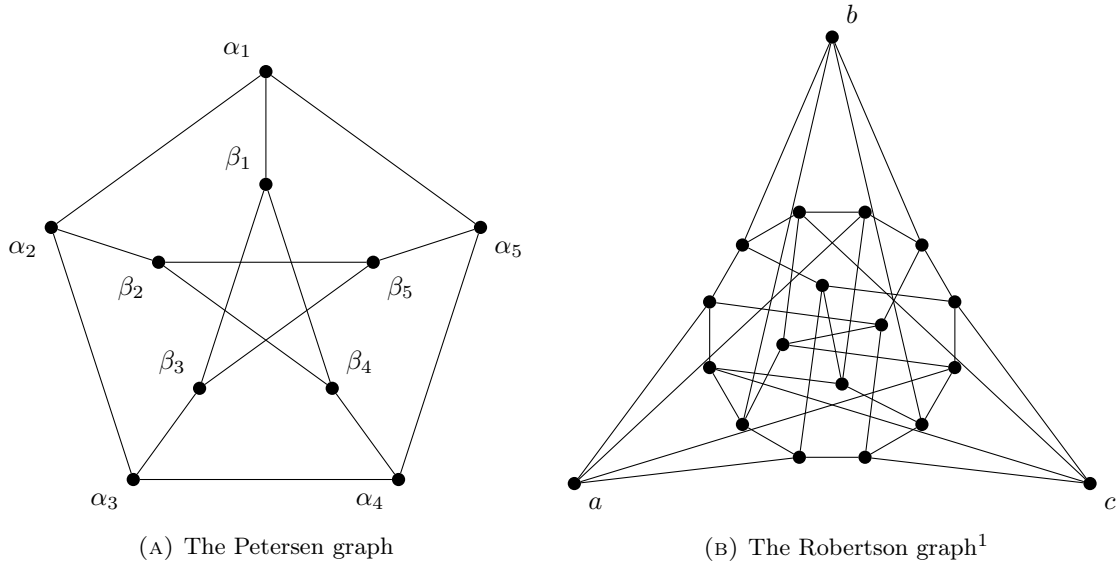
\begin{figure}[h]
	\begin{subfigure}[t]{.49\textwidth}
		\centering
		\begin{tikzpicture}[scale=3, dot/.style = {circle, fill, minimum
						size=#1, inner sep=0pt, outer sep=0pt}, dot/.default =
						5pt  
			]
			\node [dot] (0) at (0, 1) [label={[scale=1]above left:
			$\alpha_1$}]{}; \node [dot] (2) at (-0.951057, 0.309017)
			[label={[scale=1]below left: $\alpha_2$}] {}; \node [dot] (8) at
			(-0.587785,-0.809017)  [label={[scale=1]below left: $\alpha_3$}]{};
			\node [dot] (9) at (0.587785,-0.809017)  [label={[scale=1]below
			left: $\alpha_4$}]{}; \node [dot] (3) at (0.951057,0.309017)
			[label={[scale=1]below right: $\alpha_5$}]  {};

			\node [dot] (1) at (0,0.5*1) [label={[scale=1]above left:
			$\beta_1$}]  {}; \node [dot] (4) at (0.5*-0.951057, 0.5*0.309017)
			[label={[scale=1]below left: $\beta_2$}]  {}; \node [dot] (6) at
			(0.5*-0.587785,0.5*-0.809017) [label={[scale=1]above left:
			$\beta_3$}] {}; \node [dot] (7) at (0.5*0.587785,0.5*-0.809017)
			[label={[scale=1]above right: $\beta_4$}] {}; \node [dot] (5) at
			(0.5*0.951057, 0.5*0.309017)  [label={[scale=1]below right:
			$\beta_5$}] {};

			\draw (0) to (1); \draw (0) to (2); \draw (0) to (3);

			\draw (1) to (6); \draw (1) to (7);

			\draw (2) to (8); \draw (2) to (4);

			\draw (4) to (5); \draw (4) to (7);

			\draw (5) to (6);

			\draw (8) to (6); \draw (8) to (9);

			\draw (9) to (7);

			\draw (9) to (3); \draw (3) to (5);
		\end{tikzpicture}
		\caption{The Petersen graph}
	\end{subfigure}
	~
	\begin{subfigure}[t]{.49\textwidth}
		\centering
		\begin{tikzpicture}[scale=2.5, dot/.style = {circle, fill, minimum
						size=#1, inner sep=0pt, outer sep=0pt}, dot/.default =
						5pt
			]


			\node [dot] (1) at (2.9787, -1.1193)[label={below right:$c$}] {};
			\node [dot] (2) at (0.9558, -0.502) {}; \node [dot] (3) at (1.1304,
			-0.8047) {}; \node [dot] (4) at (1.3454, -0.3795) {}; \node [dot]
			(5) at (1.4329, 0.3248) {}; \node [dot] (6) at (1.1304, 0.1501) {};
			\node [dot] (7) at (1.6076, 1.2567)[label={above right:$b$}] {};
			\node [dot] (8) at (2.0847, -0.8047) {}; \node [dot] (9) at (1.6598,
			-0.5896) {}; \node [dot] (10) at (1.7822, 0.3248) {}; \node [dot]
			(11) at (2.0847, 0.1501) {}; \node [dot] (12) at (1.8697, -0.2751)
			{}; \node [dot] (13) at (0.9558, -0.1526) {}; \node [dot] (14) at
			(0.2364, -1.1193)[label={below right:$a$}] {}; \node [dot] (15) at
			(2.2594, -0.502) {}; \node [dot] (16) at (2.2594, -0.1526) {}; \node
			[dot] (17) at (1.5554, -0.065) {}; \node [dot] (18) at (1.4329,
			-0.9794) {}; \node [dot] (19) at (1.7822, -0.9794) {};

			\draw (1) to (2); \draw (1) to (5); \draw (1) to (16); \draw (1) to
			(19); \draw (2) to (3); \draw (2) to (9); \draw (2) to (13); \draw
			(3) to (4); \draw (3) to (7); \draw (3) to (18); \draw (4) to (5);
			\draw (4) to (12); \draw (4) to (15); \draw (5) to (6); \draw (5) to
			(10); \draw (6) to (7); \draw (6) to (13); \draw (6) to (17); \draw
			(7) to (8); \draw (7) to (11); \draw (8) to (9); \draw (8) to (15);
			\draw (8) to (19); \draw (9) to (10); \draw (9) to (17); \draw (10)
			to (11); \draw (10) to (14); \draw (11) to (12); \draw (11) to (16);
			\draw (12) to (13); \draw (12) to (19); \draw (13) to (14); \draw
			(14) to (15); \draw (14) to (18); \draw (15) to (16); \draw (16) to
			(17); \draw (17) to (18); \draw (18) to (19);
		\end{tikzpicture}
		\caption{The Robertson graph\footnotemark}
	\end{subfigure}
	\caption{Some small $(d,5)$-cage graphs}
	\label{fig:petersenrobertson_graph}
\end{figure}
\footnotetext{Computer-generated drawing
	\cite{wolfram_research_inc_mathematica_nodate}.} The main result of this
	article is to confirm that $M_4=19$. Although we are not able to prove a
	complete uniqueness result, we narrow down the possible 4-cop-win graphs on
	19 vertices by proving there are no 4-cop-win graphs on 19 vertices with
	maximum degree at most 4 or at least 7, except for the Robertson graph.

Although our proof is not directly based on those in \cite{hosseini_game_2018}
and \cite{baird_minimum_2014} (we however use two results from the latter), there are
certainly some common elements between those arguments and ours. In particular,
we similarly break down the problem by maximum degree and gradually develop
properties of potential 4-cop-win graphs by constructing explicit winning
strategies for 3 cops if those properties do not hold.

While we are able to obtain many interesting results using structural
properties, this article makes extensive use of computational methods to verify
the remaining cases. All of the code and data produced in the writing of this
article is available online at \cite{turcotte_code_2020}. This includes not only
the final results, but the graphs we generate in the intermediate algorithms,
precise counts of the number of graphs we generate at every step in these
algorithms, and the time required for almost all computations. All of the
computations are split up in small parts to facilitate verification. At various
points in this article, we will also discuss possible improvements and
alternative computational approaches.

The structure of this article is as follows. In Section \ref{previoussection},
we present the previously known results which we will need in our proof. In
Section \ref{smallgraphssections}, we compute the 3-cop-win graphs on 11 to 14
vertices (with some maximum degree restrictions), which we will use in later
sections to construct all graphs which are possibly 4-cop-win. In Section
\ref{highdegreesection}, by considering graphs containing the Petersen graph, we
show that all graphs on $n\leq 18$ vertices with maximum degree $n-12$ or $n-11$
have at most 3 vertices. By refining the techniques used for $n\leq 18$, we will
be able to expand the result to $n=19$. In Section \ref{maxdeg3section}, we
present the results of an exhaustive search which shows there are no 4-cop-win
subcubic graphs on $n\leq 20$ vertices, and mention an application to a related
problem. In Section \ref{remainingcasessection}, we introduce the Merging
Algorithm in order to tackle the cases with maximum degree $3<\Delta<n-12$.
These are be the last remaining cases, and will allow us to prove in Section
\ref{mainsection} that the minimum order of 4-cop-win graphs is 19. A visual aid
to how each part will be used in the proof is provided in Table
\ref{tab:applythms}.

\section{Notation and previous results}\label{previoussection} In this section,
we introduce most of the notation used in the article. We also cite previously
known results that will be useful in the following sections.

When considering a graph $G$, we will respectively denote by $n(G)$, $d_G(u)$,
$\delta(G)$ and $\Delta(G)$ the number of vertices of $G$, the degree of a
vertex $u$ in $G$, the minimum degree of $G$ and the maximum degree of $G$. If
$u$ is a vertex of $G$, $N_G(u)$ will denote the (open) neighbourhood of $u$ and
$N_G[u]=N_G(u)\cup \{u\}$ will denote the closed neighbourhood of $u$. For these
symbols, we will usually omit the $G$ when the choice of graph is easily
deduced.

For $S\subseteq V(G)$, $S^c$ will denote the complement (relative to $V(G)$) of
$S$, $\langle S \rangle$ will denote the subgraph of $G$ induced by $S$, and
$G-S$ will denote $\langle S^c\rangle$. When $S=\{x\}$, we will use the notation
$G-x$ instead of $G-S$. Similarly, if $H$ is a subgraph of $G$, then $G-H$ will
denote $G-V(H)$. When $G$ is isomorphic to a graph $G'$, we will write $G \simeq
G'$.

Denote by $\peter_0$ the Petersen graph, as seen in Figure
\ref{fig:petersenrobertson_graph}. As $\peter_0$ is $3$-regular with girth 5, we
know that $c(\peter_0)\geq 3$, and as it contains a dominating set of size $3$,
we know that $c(G)=3$. As noted in the introduction, the following theorem was
stated by Andreae in \cite{andreae_pursuit_1986} and proved by Baird et al.,
first by computer verification, as well as by using structural properties.

\begin{theorem}\cite{andreae_pursuit_1986, baird_minimum_2014}\label{min3} Let
	$G$ be a connected graph.
	\begin{enumerate}
		\item If $n\leq 9$, then $c(G)\leq 2$.
		\item If $n=10$, then $c(G)\leq 2$, unless $G$ is the Petersen graph.
	\end{enumerate}
	In particular, $M_3=10$.
\end{theorem}

The proof of the previous theorem makes use of the following lemma, which will
also be useful to us.

\begin{lemma}\label{nminus5degree}\cite{baird_minimum_2014} Let $G$ be a
	connected graph. If $\Delta \geq n-5$, then $c(G) \leq 2$.
\end{lemma}

A simple, visual proof of this lemma is available in
\cite{beveridge_petersen_2012}. We now introduce the concept of a retract, a
tool often used to study the game of cops and robbers (see for example
\cite{aigner_game_1984,nowakowski_vertex--vertex_1983}). The following is based
on \cite{berarducci_cop_1993}.

\begin{definition}
	Let $G$ be a graph. If $H$ is an induced subgraph of $G$, we say $H$ is a
	\emph{retract} of $G$ if there exists a mapping $f:V(G)\rightarrow V(H)$
	such that:
	\begin{enumerate}
		\item If $xy\in E(G)$, then $f(x)f(y)\in E(H)$ or $f(x)=f(y)$.
		\item $f|_{V(H)}: V(H)\rightarrow V(H)$ is the identity mapping.
	\end{enumerate}
	Such a mapping $f$ is called a \emph{retraction}.
\end{definition}
This definition formalizes the intuitive idea that $G$ can be "folded" onto $H$,
where each edge must either be sent onto an edge or onto a vertex. Those
familiar with graph homomorphisms will notice that condition (1) states that $f$
is a homomorphism from $G$ to $H$ if we consider $H$ to be reflexive (that is,
if we add a loop at each vertex of $H$). Reflexivity is necessary as consequence
of allowing the cops and the robber to stay on a vertex at their turn, implying
a loop on each vertex. The concept of retracts has been central in the study of
the game of cops and robbers, for instance also appearing in
\cite{nowakowski_vertex--vertex_1983}.

If $G$ is disconnected, denote $G_1,\dots,G_t$ the connected components of $G$.
By extension, we may define the cop number of a disconnected graph by
$c(G)=\max_{1\leq i\leq t}c(G_i)$. These definitions allow us to state the
following result of Berarduci and Intriglia, which we will use many times to
reduce the number of cases we need to consider.

\begin{theorem}\label{generalretract}\cite{berarducci_cop_1993} If $G$ is a
	connected graph and $H$ is a retract of $G$, then $$c(H)\leq c(G)\leq \max
	\{c(H),c(G-H)+1\}.$$
\end{theorem}

A specific case of this theorem is the following reformulation of a corollary in
\cite{berarducci_cop_1993}, which will often be easier to use.

\begin{corollary}\label{retractcomponents} If $G$ is a connected graph, $u$ is a
	vertex of $G$ and $K$ is a union of some connected components of $G-N[u]$,
	then
	$$c(G-K)\leq c(G)\leq \max\{c(G-K),c(K)+1\}.$$

	In particular, if $c(K)\leq k-1$, then $c(G)\leq k$ if and only if
		$c(G-K)\leq k$.
\end{corollary}
\begin{proof}
	It is easy to verify that $f:V(G)\rightarrow V(G-K)$ defined by
	$$f(x)=\begin{cases} u & \text{if }x\in V(K) \\
			x & \text{otherwise}\end{cases}$$ is a retraction. It is only left
	to apply Theorem \ref{generalretract} to $H=G-K$.
\end{proof}

One trivial consequence of this corollary is that if the cop number of every
component of $G-N[u]$ is at most $k-1$, then $c(G)\leq k$. One can also see this
directly by leaving a fixed cop on $u$ and playing with $k-1$ cops on the
connected component of $G-N[u]$ in which the robber is located.

We then easily get the following result, which is implicit in
\cite{hosseini_game_2018}.

\begin{corollary}\label{maxdegnlarger11} If $G$ is a connected graph and
	$\Delta>n-11$, then $c(G)\leq 3$.
\end{corollary}
\begin{proof}
	If $\Delta>n-11$ and $u$ is a vertex of maximum degree, then
	$|V(G-N[u])|<10$. By Theorem \ref{min3}, every connected component of
	$G-N[u]$ has cop number at most 2. The last remark yields the result.
\end{proof}

Finally, we recall a well known concept in the study of the game of cops and
robbers.

\begin{definition}
	Let $x,u$ be distinct vertices of a graph $G$. If $N(x)\subseteq N[u]$, we
	say $x$ is \emph{cornered} by $u$ or that $x$ is a \emph{corner}.
\end{definition}

We note that this is a slight variation on the classical notion of a corner (or
irreducible vertex), as it normally requires $ux$ to be an edge, see
\cite{nowakowski_vertex--vertex_1983}. We may now get the following well-known
result as a further simplification of Corollary \ref{retractcomponents}.

\begin{corollary}\label{removecorner} Let $G$ be a connected graph and $x$ be a
	corner of $G$. If $c(G-x)\geq 2$, then $c(G)=c(G-x)$. If $c(G-x)=1$, then
	$c(G)\in \{1,2\}$.
\end{corollary}
\begin{proof}
	Let $u$ be a vertex cornering $x$. First define $f:V(G)\rightarrow V(G-x)$
	by $$f(y)=\begin{cases} u & \text{if }y=x     \\
			y & \text{otherwise.}\end{cases}$$It is easily seen that the
			condition that $x$ be cornered by $u$ implies that $f$ is a
			retraction. The statement then follows from Theorem
			\ref{generalretract}.
\end{proof}


\section{Computational results for small 3-cop-win
  graphs}\label{smallgraphssections} In this section, we find some 3-cop-win
  graphs on at most $14$ vertices respecting some degree conditions. We will do
  this by computing the cop number of every graph of the desired orders and
  degrees. The results will be useful in the following sections.

Graph generation in this section and in Section \ref{maxdeg3section} is done
using the \texttt{geng} function provided with the nauty/Traces package (version
26r12) \cite{mckay_practical_2014}. The algorithm to compute the cop number is
similar to that proposed, in particular, in
\cite{bonato_cops_2010,clarke_characterizations_2012,rickert_cops_2017}, which
we have implemented for cop numbers $1,2,3$ in the Julia language
\cite{bezanson_julia_2017,bromberger_juliagraphslightgraphsjl_2017}. For a given
$k$ (which will be between $1$ and $3$ in our case), the algorithm determines
whether $c(G) \leq k$ or $c(G)>k$. Hence, if we run the algorithms for
$1,\dots,m$ cops on a graph, we can determine the cop number exactly if it is a
most $m$, or if not then say that $c(G)\geq m+1$.

To test the validity of our implementation, we have compared the results for the
cop number of connected graphs up to 10 vertices to those in
\cite{baird_minimum_2014}. Following a small discrepancy between the counts, our
tallies of cop-win graphs were also verified to be correct by implementing a
dismantling algorithm \cite{nowakowski_vertex--vertex_1983} and by comparing
with the implementation at \cite{afanassiev_cop-number_2017}. To test our code
for higher cop numbers, it was also run on some cage graphs which we know 3 cops
lose. Based on the results of these tests, we are confident in the correctness
of our implementation.

We first define a variant of the Petersen graph.
\begin{definition}
	We say a connected graph $G$ is a \emph{cornered Petersen graph} if $G$
	contains a corner $m$ such that $G-m\simeq \peter_0$. There are 6 such
	graphs up to isomorphism. We denote them $\peter_i$, $i=1,\dots,6$, as seen
	in Figure \ref{fig:cornered_petersen_graphs}.
\end{definition}

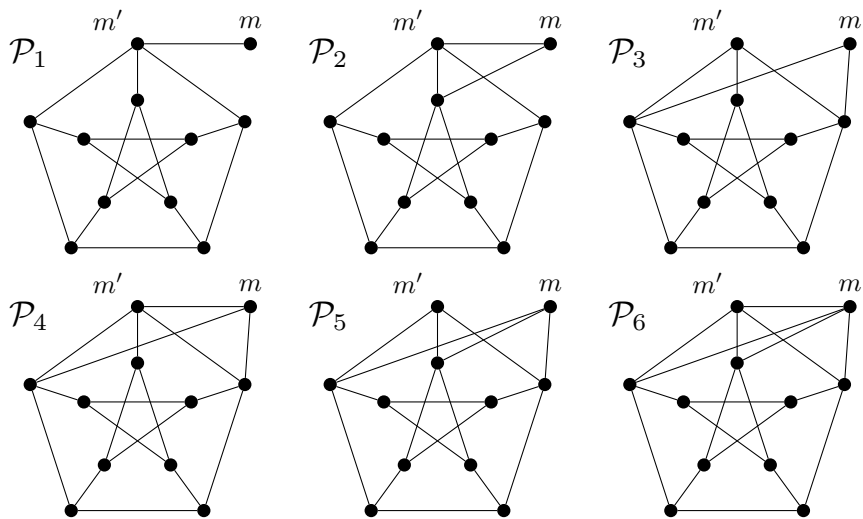
\begin{figure}[h]
	\centering
	\begin{tabular}{ccc}
		\begin{tikzpicture}[scale=1.5, dot/.style = {circle, fill, minimum
						size=#1, inner sep=0pt, outer sep=0pt}, dot/.default =
						5pt
			]
			\node [dot] (0) at (0, 1) [label={[scale=1]above left: $m'$}]{};
			\node [dot] (2) at (-0.951057, 0.309017)  {}; \node [dot] (8) at
			(-0.587785,-0.809017)  {}; \node [dot] (9) at (0.587785,-0.809017)
			{}; \node [dot] (3) at (0.951057,0.309017)   {};

			\node [dot] (1) at (0,0.5*1)   {}; \node [dot] (4) at
			(0.5*-0.951057, 0.5*0.309017)   {}; \node [dot] (6) at
			(0.5*-0.587785,0.5*-0.809017) {}; \node [dot] (7) at
			(0.5*0.587785,0.5*-0.809017)  {}; \node [dot] (5) at (0.5*0.951057,
			0.5*0.309017)   {};

			\node [dot] (10) at (1, 1) [label={[scale=1] above: $m$}]{}; \node
			[] (11) at (-0.6, 0.6) [label={[scale=1.3] above left:
			$\peter_1$}]{};

			\draw (0) to (1); \draw (0) to (2); \draw (0) to (3);

			\draw (1) to (6); \draw (1) to (7);

			\draw (2) to (8); \draw (2) to (4);

			\draw (4) to (5); \draw (4) to (7);

			\draw (5) to (6);

			\draw (8) to (6); \draw (8) to (9);

			\draw (9) to (7);

			\draw (9) to (3); \draw (3) to (5);

			\draw (10) to (0);
		\end{tikzpicture}
		 &
		\begin{tikzpicture}[scale=1.5, dot/.style = {circle, fill, minimum
						size=#1, inner sep=0pt, outer sep=0pt}, dot/.default =
						5pt
			]
			\node [dot] (0) at (0, 1) [label={[scale=1]above left: $m'$}]{};
			\node [dot] (2) at (-0.951057, 0.309017)  {}; \node [dot] (8) at
			(-0.587785,-0.809017)  {}; \node [dot] (9) at (0.587785,-0.809017)
			{}; \node [dot] (3) at (0.951057,0.309017)   {};

			\node [dot] (1) at (0,0.5*1)   {}; \node [dot] (4) at
			(0.5*-0.951057, 0.5*0.309017)   {}; \node [dot] (6) at
			(0.5*-0.587785,0.5*-0.809017)  {}; \node [dot] (7) at
			(0.5*0.587785,0.5*-0.809017)  {}; \node [dot] (5) at (0.5*0.951057,
			0.5*0.309017)   {};

			\node [dot] (10) at (1, 1) [label={[scale=1] above: $m$}]{}; \node
			[] (11) at (-0.6, 0.6) [label={[scale=1.3] above left:
			$\peter_2$}]{};

			\draw (0) to (1); \draw (0) to (2); \draw (0) to (3);

			\draw (1) to (6); \draw (1) to (7);

			\draw (2) to (8); \draw (2) to (4);

			\draw (4) to (5); \draw (4) to (7);

			\draw (5) to (6);

			\draw (8) to (6); \draw (8) to (9);

			\draw (9) to (7);

			\draw (9) to (3); \draw (3) to (5);

			\draw (10) to (0); \draw (10) to (1);
		\end{tikzpicture}
		 &
		\begin{tikzpicture}[scale=1.5, dot/.style = {circle, fill, minimum
						size=#1, inner sep=0pt, outer sep=0pt}, dot/.default =
						5pt
			]
			\node [dot] (0) at (0, 1) [label={[scale=1]above left: $m'$}]{};
			\node [dot] (2) at (-0.951057, 0.309017)  {}; \node [dot] (8) at
			(-0.587785,-0.809017)  {}; \node [dot] (9) at (0.587785,-0.809017)
			{}; \node [dot] (3) at (0.951057,0.309017)   {};

			\node [dot] (1) at (0,0.5*1)   {}; \node [dot] (4) at
			(0.5*-0.951057, 0.5*0.309017)   {}; \node [dot] (6) at
			(0.5*-0.587785,0.5*-0.809017)  {}; \node [dot] (7) at
			(0.5*0.587785,0.5*-0.809017)  {}; \node [dot] (5) at (0.5*0.951057,
			0.5*0.309017)   {};

			\node [dot] (10) at (1, 1) [label={[scale=1] above: $m$}]{}; \node
			[] (11) at (-0.6, 0.6) [label={[scale=1.3] above left:
			$\peter_3$}]{};

			\draw (0) to (1); \draw (0) to (2); \draw (0) to (3);

			\draw (1) to (6); \draw (1) to (7);

			\draw (2) to (8); \draw (2) to (4);

			\draw (4) to (5); \draw (4) to (7);

			\draw (5) to (6);

			\draw (8) to (6); \draw (8) to (9);

			\draw (9) to (7);

			\draw (9) to (3); \draw (3) to (5);

			\draw (10) to (2); \draw (10) to (3);
		\end{tikzpicture}
		\\
		\begin{tikzpicture}[scale=1.5, dot/.style = {circle, fill, minimum
						size=#1, inner sep=0pt, outer sep=0pt}, dot/.default =
						5pt
			]
			\node [dot] (0) at (0, 1) [label={[scale=1]above left: $m'$}]{};
			\node [dot] (2) at (-0.951057, 0.309017)  {}; \node [dot] (8) at
			(-0.587785,-0.809017)  {}; \node [dot] (9) at (0.587785,-0.809017)
			{}; \node [dot] (3) at (0.951057,0.309017)   {};

			\node [dot] (1) at (0,0.5*1)   {}; \node [dot] (4) at
			(0.5*-0.951057, 0.5*0.309017)   {}; \node [dot] (6) at
			(0.5*-0.587785,0.5*-0.809017)  {}; \node [dot] (7) at
			(0.5*0.587785,0.5*-0.809017)  {}; \node [dot] (5) at (0.5*0.951057,
			0.5*0.309017)  {};

			\node [dot] (10) at (1, 1) [label={[scale=1] above: $m$}]{}; \node
			[] (11) at (-0.6, 0.6) [label={[scale=1.3] above left:
			$\peter_4$}]{};

			\draw (0) to (1); \draw (0) to (2); \draw (0) to (3);

			\draw (1) to (6); \draw (1) to (7);

			\draw (2) to (8); \draw (2) to (4);

			\draw (4) to (5); \draw (4) to (7);

			\draw (5) to (6);

			\draw (8) to (6); \draw (8) to (9);

			\draw (9) to (7);

			\draw (9) to (3); \draw (3) to (5);

			\draw (10) to (0); \draw (10) to (2); \draw (10) to (3);
		\end{tikzpicture}
		 &
		\begin{tikzpicture}[scale=1.5, dot/.style = {circle, fill, minimum
						size=#1, inner sep=0pt, outer sep=0pt}, dot/.default =
						5pt
			]
			\node [dot] (0) at (0, 1) [label={[scale=1]above left: $m'$}]{};
			\node [dot] (2) at (-0.951057, 0.309017)  {}; \node [dot] (8) at
			(-0.587785,-0.809017)  {}; \node [dot] (9) at (0.587785,-0.809017)
			{}; \node [dot] (3) at (0.951057,0.309017)   {};

			\node [dot] (1) at (0,0.5*1)   {}; \node [dot] (4) at
			(0.5*-0.951057, 0.5*0.309017)   {}; \node [dot] (6) at
			(0.5*-0.587785,0.5*-0.809017) {}; \node [dot] (7) at
			(0.5*0.587785,0.5*-0.809017)  {}; \node [dot] (5) at (0.5*0.951057,
			0.5*0.309017)   {};

			\node [dot] (10) at (1, 1) [label={[scale=1] above: $m$}]{}; \node
			[] (11) at (-0.6, 0.6) [label={[scale=1.3] above left:
			$\peter_5$}]{};

			\draw (0) to (1); \draw (0) to (2); \draw (0) to (3);

			\draw (1) to (6); \draw (1) to (7);

			\draw (2) to (8); \draw (2) to (4);

			\draw (4) to (5); \draw (4) to (7);

			\draw (5) to (6);

			\draw (8) to (6); \draw (8) to (9);

			\draw (9) to (7);

			\draw (9) to (3); \draw (3) to (5);

			\draw (10) to (1); \draw (10) to (2); \draw (10) to (3);
		\end{tikzpicture}
		 &
		\begin{tikzpicture}[scale=1.5, dot/.style = {circle, fill, minimum
						size=#1, inner sep=0pt, outer sep=0pt}, dot/.default =
						5pt
			]
			\node [dot] (0) at (0, 1) [label={[scale=1]above left: $m'$}]{};
			\node [dot] (2) at (-0.951057, 0.309017)  {}; \node [dot] (8) at
			(-0.587785,-0.809017)  {}; \node [dot] (9) at (0.587785,-0.809017)
			{}; \node [dot] (3) at (0.951057,0.309017)   {};

			\node [dot] (1) at (0,0.5*1)   {}; \node [dot] (4) at
			(0.5*-0.951057, 0.5*0.309017)   {}; \node [dot] (6) at
			(0.5*-0.587785,0.5*-0.809017) {}; \node [dot] (7) at
			(0.5*0.587785,0.5*-0.809017)  {}; \node [dot] (5) at (0.5*0.951057,
			0.5*0.309017)   {};

			\node [dot] (10) at (1, 1) [label={[scale=1] above: $m$}]{}; \node
			[] (11) at (-0.6, 0.6) [label={[scale=1.3] above left:
			$\peter_6$}]{};

			\draw (0) to (1); \draw (0) to (2); \draw (0) to (3);

			\draw (1) to (6); \draw (1) to (7);

			\draw (2) to (8); \draw (2) to (4);

			\draw (4) to (5); \draw (4) to (7);

			\draw (5) to (6);

			\draw (8) to (6); \draw (8) to (9);

			\draw (9) to (7);

			\draw (9) to (3); \draw (3) to (5);

			\draw (10) to (0); \draw (10) to (1); \draw (10) to (2); \draw (10)
			to (3);
		\end{tikzpicture}
	\end{tabular}

	\caption{The cornered Petersen graphs}
	\label{fig:cornered_petersen_graphs}
\end{figure}

We now solve a question raised in \cite{baird_minimum_2014}, classifying the
3-cop-win graphs on 11 vertices, albeit computationally.

\begin{proposition}\label{classification11} If $G$ is a connected graph such
	that $n=11$, then $c(G)=3$ if and only if $G\simeq \peter_i$ for some $1\leq
	i\leq 6$. Otherwise, $c(G)\leq 2$.
\end{proposition}

\begin{proof}
	Firstly, it is clear by Theorem \ref{min3} and Corollary \ref{removecorner}
	that the cornered Petersen graphs are 3-cop-win. We would like to show that
	these graphs are the only graphs on 11 vertices with cop number 3, and that
	all other graphs have cop number at most 2.

	By Lemma \ref{nminus5degree}, we may only consider graphs such that
		$\Delta\leq n-6=5$. We generate all connected graphs on 11 vertices such
		that $\Delta\leq 5$ and classify each graph according to its cop number
		(by running the algorithm for $k=1,2$). The results are presented in
		Table \ref{table:n11-14} (the counts are up to isomorphism). The 6
		graphs found are the graphs $\peter_i$ for $i=1,\dots,6$, which
		concludes the proof.
\end{proof}

This is an interesting phenomenon: the unique 3-cop-win graph on 10 vertices is
a retract of all the $3$-cop-win graphs on 11 vertices. This behaviour does not
occur for the 2-cop-win graphs: the minimum 2-cop-win graph is the $4$-cycle, on
which the 5-cycle does not retract. Although we will not have any answer for
this question in this article, it would be interesting to know whether in
general (even for 4-cop-win graphs only), the $k$-cop-win graphs on $M_k+1$
vertices can be retracted on $k$-cop-win graph(s) on $M_k$ vertices.

In Section \ref{remainingcasessection}, we will also need the following lemma.

\begin{lemma}\label{computationn1213} There exist
	\begin{itemize}
		\item $80$ connected graphs $G$ on $12$ vertices with $\Delta\leq 4$,
		\item $173$ connected graphs $G$ on $12$ vertices with $\Delta\leq 5$,
		\item $1105$ connected graphs $G$ on $13$ vertices with $\Delta\leq 4$,
		and
		\item $16523$ connected graphs $G$ on $14$ vertices with $\Delta\leq 4$
	\end{itemize}
	such that $c(G) = 3$. All other connected graphs $G$ considered with these
	orders and maximum degrees satisfy $c(G)\leq 2$.
\end{lemma}

\begin{proof}
	Firstly, all graphs on at most 14 vertices have cop number at most 3. For
	cases where $\Delta \geq 4$, this is a direct consequence of Corollary
	\ref{maxdegnlarger11}. For $\Delta=2$, the graph is either a path or a
	cycle. For $\Delta=3$, see the results of Table \ref{table:deg3}. This is
	also a direct consequence of knowing that $M_4 \geq 16$, see
	\cite{hosseini_game_2018}.

	We generate, up to isomorphism, all connected graphs on $12$ vertices such
	that $\Delta\leq 5$ and on $13$ and $14$ vertices such that $\Delta\leq 4$.
	We classify these graphs according to their cop number (by running the
	algorithm for $k=1,2$). Afterwards, we also count which of the graphs on 12
	vertices with $\Delta\leq 5$ and $c(G)\geq 3$ are such that $\Delta\leq 4$.
	The results are in Table \ref{table:n11-14}.
\end{proof}

\begin{table}[h]
	\begin{tabular}{|c|c|c|c|c|c|}
		\cline{4-6}
		\multicolumn{3}{c}{} & \multicolumn{3}{|c|}{Cop number}
		\\
		\hline
		\rowcolor{gray!60}
		$n$                  & Degree bounds                    & Number of
		graphs & $1$    & $2$       & $\geq 3$ \\ \hline
		$11$                 & $\Delta\leq 5$                   & 21503340
		& 69310  & 21434024  & 6        \\ \hline
		$12$                 & $\Delta\leq 4$                   & -
		& -      & -         & 80       \\ \hline
		$12$                 & $\Delta\leq 5$                   & 471142472
		& 295377 & 470846922 & 173      \\ \hline
		$13$                 & $\Delta\leq 4$                   & 68531618
		& 73876  & 68456637  & 1105     \\ \hline
		$14$                 & $\Delta\leq 4$                   & 748592936
		& 247022 & 748329391 & 16523    \\ \hline
	\end{tabular}
	\caption{Cop number breakdown for connected graphs on 11-14 vertices with some degree restrictions}
	\label{table:n11-14}
\end{table}

While the counts are presented to summarize the results, the precise 3-cop-win
graphs are the focus of our attention as we use them in the following sections.

We now discuss some possible improvements to our approach. This part is not
essential to prove the main result, but may be helpful in motivating the methods
of Section \ref{remainingcasessection}. To achieve these results, we
exhaustively computed the cop number of every connected graph that satisfied our
maximum degree constraints. Since we proceeded by exhaustion, the run time of
these computations is somewhat long due to the high number of graphs, especially
in the case of $\Delta = 14$. We note that a more clever approach might yield
faster calculation time.

The first and most obvious improvement would be to only look at graphs with a
minimum degree of at least 2, which can reduce the number of graphs to consider
by up to around 50\%. If we already know the 3-cop-win graphs on one fewer
vertices, we can then just consider all possible ways to attach an extra vertex
of degree 1 to those graphs. Using this method, we can get all connected
3-cop-win graphs of a given order.

However, this method is still an exhaustive search. A more clever approach would
be to consider every (not necessarily connected) 2-cop-win graph $G'$ on
$n-\Delta-1$ vertices, add a vertex $u$ with $\Delta$ neighbours new neighbours
and consider each way of adding edges between $N(u)$ and $G'$ (up to
isomorphism), then checking which of these graphs are 3-cop-win. We would
recommend the interested reader try this approach.

A more refined approach of this would be to use the algorithm of Section
\ref{remainingcasessection} to build candidate 3-cop-win graphs, by merging
2-cop-win graphs on fewer vertices. We will see later that although this method
can reduce significantly the computation time, in practice it requires some
effort to make sure all the possible cases are considered. For the size of
graphs we are considering, this approach does not appear worthwhile.


\section{Graphs with high maximum degree}\label{highdegreesection} In this
section, we consider the cop number of graphs $G$ such that $\Delta=n-11$ or
$\Delta=n-12$. We start by investigating some properties of the game of cops and
robbers on the Petersen graph (denoted by $\peter_0$) and its variants, the
cornered Petersen graphs (denoted by $\peter_i$ for $1\leq i\leq 6$). Many of
the arguments in this section are extremely simple once visualized. For this
reason, we have provided many figures representing parts of the proofs. Of
course, we cannot provide figures for every case, so we encourage the reader
draw out the graphs while following the proofs, especially regarding the
movements of the cops and the robber.

By considering the Petersen graph as the Kneser graph $KG_{5,2}$
\cite{baez_petersen_2015}, one easily gets the following well-known result
(although maybe not with this precise formulation), an illustration of the fact
that the Petersen graph is highly transitive.

\begin{definition}
	\label{definition:strong_stable_set}
	We say a set of 3 vertices $\{x,y,z\}$ is a \emph{strong stable set} if it
	is a stable set and if $N(x)\cap N(y)\cap N(z)=\emptyset$.
\end{definition}

\begin{lemma}
	\label{petersenautomorphisms}
	\item
	\begin{enumerate}[label=(\alph*)]
		\item If $\{x,y,z\}$ and $\{x',y',z'\}$ are strong stable sets of
		      $\peter_0$, then there exists an automorphism $\phi_1$ of
		      $\peter_0$ such that $\phi_1(x)=x'$, $\phi_1(y)=y'$ and
		      $\phi_1(z)=z'$.
		\item If $ab$ and $a'b'$ are two edges of $\peter_0$, then there exists
		      an automorphism $\phi_2$ of $\peter_0$ such that $\phi_2(a)=a'$,
		      $\phi_2(b)=b'$. This property is known as being arc-transitive.
	\end{enumerate}
\end{lemma}

We use the labels $m$, $m'$ on the graphs $\peter_i$ for $i=1,\dots,6$, as shown
in Figure \ref{fig:cornered_petersen_graphs}. In particular, for each of these
graphs, $\peter_i-m\simeq \peter_0$. We also see that $m'$ always corners $m$,
which will be very useful. We also note that as $m,m'\notin V(\peter_0)$, we can
say that $\peter_0-m=\peter_0-m'=\peter_0$.

As stated in Theorem \ref{min3}, we know that $c(\peter_0)=3$. In the next two
lemmas, we show that although two cops do not have a winning strategy, they have
a lot of power as to which positions can be reached. We will later use these
lemmas as parts of strategies on graphs containing the Petersen or the cornered
Petersen graph. These lemmas would be very easy to establish computationally,
but we consider a proof to be worthwhile.

\begin{lemma}\label{chasinglemma1} Let $0\leq i\leq 6$. If $\{x,y,z\}$ is a
	strong stable set of $\peter_i-m$, then there exists a strategy for 2 cops
	on $\peter_i$ to reach the following situation.
	\begin{enumerate}
		\item The robber is on $x$, except  in the case $x=m'$ and
		      $i\in\{5,6\}$, where the robber is either on $m'$ or $m$.
		\item The cops are on $y$ and $z$.
		\item It is the cops' turn.
	\end{enumerate}
\end{lemma}
\begin{proof}
	Let us first consider the case of $\peter_0$. Consider the labelling of
	$\peter_0$ with $\alpha_i$ and $\beta_i$ as shown in Figure
	\ref{fig:petersenrobertson_graph}. For any $j,j'$, observe that if two cops
	are on $\beta_{j}, \beta_{j+1}$ (working in modulo 5), they can directly
	move to the pair $\beta_{j'},\beta_{j'+1}$.

	Without loss of generality, we may consider that $x=\alpha_1, y=\beta_2,
		z=\beta_3$, as for all other strong stable sets we can apply the
		automorphism of Lemma \ref{petersenautomorphisms}(a). For some $k$, we
		start the game with two cops on $\beta_k$ and $\beta_{k+1}$ (modulo 5).
		The robber must choose one of the $\alpha_\ell$ as its starting vertex.
		Notice that if the robber is on $\alpha_j$, moving the cops to $\beta_j$
		and $\beta_{j+1}$ forces the robber to move to $\alpha_{j-1}$. By
		repeating this strategy, the cops can essentially make the robber turn
		in circles on the outer 5-cycle of $\peter_0$. At the end of every cops'
		turn (except the first), the robber is on $\alpha_j$ and the cops are on
		$\beta_j$ and $\beta_{j+1}$, for some $j$. The cops repeat until the
		robber is on $\alpha_1$ (unless the robber is already on $\alpha_1$ on
		the first turn, in which case the cops make the robber do a full circle
		around the graph); it is now the cops' turn and the game is in the
		desired situation. Observe that this strategy works for any initial
		choice of $k$. This will be useful later, as for any vertex $w\in
		\peter_0$, we may choose an initial position such that one of the cops
		is in $N[w]$. We call this the \emph{chasing strategy} for the Petersen
		graph. An example is illustrated in Figure \ref{fig:chasing_strategy}.
		Even though this might be a very simple idea, this strategy is critical
		for the rest of this section as it will enable us to construct more
		complicated strategies.

	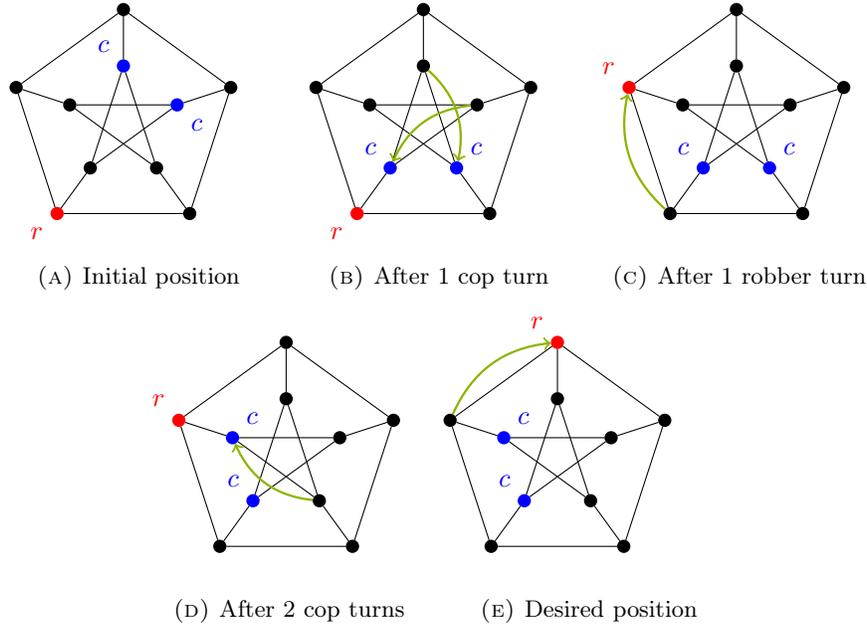
\begin{figure}[h]
		\begin{subfigure}{.25\textwidth}
			\begin{tikzpicture}[scale=1.5, dot/.style = {circle, fill, minimum
							size=#1, inner sep=0pt, outer sep=0pt}, dot/.default
							= 5pt
				]

				\node [dot] (0) at (0, 1) {}; \node [dot] (2) at (-0.951057,
				0.309017) [label={[scale=1]below left: }] {}; \node [dot,red]
				(8) at (-0.587785,-0.809017) [label={[scale=1,red]below left:
				$r$}]{}; \node [dot] (9) at (0.587785,-0.809017) {}; \node [dot]
				(3) at (0.951057,0.309017) {};

				\node [dot,blue] (1) at (0,0.5*1) [label={[scale=1,blue]above
				left: $c$}] {}; \node [dot] (4) at (0.5*-0.951057, 0.5*0.309017)
				{}; \node [dot] (6) at (0.5*-0.587785,0.5*-0.809017) {}; \node
				[dot] (7) at (0.5*0.587785,0.5*-0.809017)  {}; \node [dot,blue]
				(5) at (0.5*0.951057, 0.5*0.309017)  [label={[scale=1,blue]below
				right: $c$}] {};

				\draw (0) to (1); \draw (0) to (2); \draw (0) to (3);

				\draw (1) to (6); \draw (1) to (7);

				\draw (2) to (8); \draw (2) to (4);

				\draw (4) to (5); \draw (4) to (7);

				\draw (5) to (6);

				\draw (8) to (6); \draw (8) to (9);

				\draw (9) to (7);

				\draw (9) to (3); \draw (3) to (5);
			\end{tikzpicture}
			\caption{Initial position}
		\end{subfigure}
		\begin{subfigure}{.25\textwidth}
			\begin{tikzpicture}[scale=1.5, dot/.style = {circle, fill, minimum
							size=#1, inner sep=0pt, outer sep=0pt}, dot/.default
							= 5pt
				]

				\node [dot] (0) at (0, 1) {}; \node [dot] (2) at (-0.951057,
				0.309017) [label={[scale=1]below left: }] {}; \node [dot,red]
				(8) at (-0.587785,-0.809017) [label={[scale=1,red]below left:
				$r$}]{}; \node [dot] (9) at (0.587785,-0.809017) {}; \node [dot]
				(3) at (0.951057,0.309017) {};

				\node [dot] (1) at (0,0.5*1)   {}; \node [dot] (4) at
				(0.5*-0.951057, 0.5*0.309017)   {}; \node [dot,blue] (6) at
				(0.5*-0.587785,0.5*-0.809017)[label={[scale=1,blue]above left:
				$c$}] {}; \node [dot,blue] (7) at
				(0.5*0.587785,0.5*-0.809017)[label={[scale=1,blue]above right:
				$c$}]  {}; \node [dot] (5) at (0.5*0.951057, 0.5*0.309017)   {};

				\draw (0) to (1); \draw (0) to (2); \draw (0) to (3);

				\draw (1) to (6); \draw (1) to (7);

				\draw (2) to (8); \draw (2) to (4);

				\draw (4) to (5); \draw (4) to (7);

				\draw (5) to (6);

				\draw (8) to (6); \draw (8) to (9);

				\draw (9) to (7);

				\draw (9) to (3); \draw (3) to (5);

				\draw [->, line width=0.3mm, applegreen, bend left] (1) to (7);
					\draw [->, line width=0.3mm, applegreen, bend right] (5) to
					(6);

			\end{tikzpicture}
			\caption{After 1 cop turn}
		\end{subfigure}
		\begin{subfigure}{.25\textwidth}
			\begin{tikzpicture}[scale=1.5, dot/.style = {circle, fill, minimum
							size=#1, inner sep=0pt, outer sep=0pt}, dot/.default
							= 5pt
				]

				\node [dot] (0) at (0, 1) {}; \node [dot,red] (2) at (-0.951057,
				0.309017) [label={[scale=1,red]above left: $r$}] {}; \node [dot]
				(8) at (-0.587785,-0.809017)  {}; \node [dot] (9) at
				(0.587785,-0.809017)[label={[scale=1,red]below left:
				$\phantom{r}$}] {}; \node [dot] (3) at (0.951057,0.309017) {};

				\node [dot] (1) at (0,0.5*1)   {}; \node [dot] (4) at
				(0.5*-0.951057, 0.5*0.309017)   {}; \node [dot,blue] (6) at
				(0.5*-0.587785,0.5*-0.809017)[label={[scale=1,blue]above left:
				$c$}] {}; \node [dot,blue] (7) at
				(0.5*0.587785,0.5*-0.809017)[label={[scale=1,blue]above right:
				$c$}]  {}; \node [dot] (5) at (0.5*0.951057, 0.5*0.309017)   {};

				\draw (0) to (1); \draw (0) to (2); \draw (0) to (3);

				\draw (1) to (6); \draw (1) to (7);

				\draw (2) to (8); \draw (2) to (4);

				\draw (4) to (5); \draw (4) to (7);

				\draw (5) to (6);

				\draw (8) to (6); \draw (8) to (9);

				\draw (9) to (7);

				\draw (9) to (3); \draw (3) to (5);

				\draw [->, line width=0.3mm, applegreen, bend left] (8) to (2);

			\end{tikzpicture}
			\caption{After 1 robber turn}
		\end{subfigure}
		\begin{subfigure}{.25\textwidth}
			\begin{tikzpicture}[scale=1.5, dot/.style = {circle, fill, minimum
							size=#1, inner sep=0pt, outer sep=0pt}, dot/.default
							= 5pt
				]

				\node [dot] (0) at (0, 1) [label={[scale=1,red]above left:
				$\phantom{r}$}]{}; \node [dot,red] (2) at (-0.951057, 0.309017)
				[label={[scale=1,red]above left: $r$}] {}; \node [dot] (8) at
				(-0.587785,-0.809017)  {}; \node [dot] (9) at
				(0.587785,-0.809017)[label={[scale=1,red]below left:
				$\phantom{r}$}] {}; \node [dot] (3) at (0.951057,0.309017) {};

				\node [dot] (1) at (0,0.5*1)   {}; \node [dot,blue] (4) at
				(0.5*-0.951057, 0.5*0.309017) [label={[scale=1,blue]above right:
				$c$}]  {}; \node [dot,blue] (6) at
				(0.5*-0.587785,0.5*-0.809017)[label={[scale=1,blue]above left:
				$c$}] {}; \node [dot] (7) at (0.5*0.587785,0.5*-0.809017)  {};
				\node [dot] (5) at (0.5*0.951057, 0.5*0.309017)   {};

				\draw (0) to (1); \draw (0) to (2); \draw (0) to (3);

				\draw (1) to (6); \draw (1) to (7);

				\draw (2) to (8); \draw (2) to (4);

				\draw (4) to (5); \draw (4) to (7);

				\draw (5) to (6);

				\draw (8) to (6); \draw (8) to (9);

				\draw (9) to (7);

				\draw (9) to (3); \draw (3) to (5);

				\draw [->, line width=0.3mm, applegreen, bend left] (7) to (4);

			\end{tikzpicture}
			\caption{After 2 cop turns}
		\end{subfigure}
		\begin{subfigure}{.25\textwidth}
			\begin{tikzpicture}[scale=1.5, dot/.style = {circle, fill, minimum
							size=#1, inner sep=0pt, outer sep=0pt}, dot/.default
							= 5pt
				]

				\node [dot,red] (0) at (0, 1)[label={[scale=1,red]above left:
				$r$}] {}; \node [dot] (2) at (-0.951057, 0.309017)  {}; \node
				[dot] (8) at (-0.587785,-0.809017) {}; \node [dot] (9) at
				(0.587785,-0.809017)[label={[scale=1,red]below left:
				$\phantom{r}$}] {}; \node [dot] (3) at (0.951057,0.309017) {};

				\node [dot] (1) at (0,0.5*1)   {}; \node [dot,blue] (4) at
				(0.5*-0.951057, 0.5*0.309017) [label={[scale=1,blue]above right:
				$c$}]  {}; \node [dot,blue] (6) at
				(0.5*-0.587785,0.5*-0.809017)[label={[scale=1,blue]above left:
				$c$}] {}; \node [dot] (7) at (0.5*0.587785,0.5*-0.809017)  {};
				\node [dot] (5) at (0.5*0.951057, 0.5*0.309017)   {};

				\draw (0) to (1); \draw (0) to (2); \draw (0) to (3);

				\draw (1) to (6); \draw (1) to (7);

				\draw (2) to (8); \draw (2) to (4);

				\draw (4) to (5); \draw (4) to (7);

				\draw (5) to (6);

				\draw (8) to (6); \draw (8) to (9);

				\draw (9) to (7);

				\draw (9) to (3); \draw (3) to (5);

				\draw [->, line width=0.3mm, applegreen, bend left] (2) to (0);
			\end{tikzpicture}
			\caption{Desired position}
		\end{subfigure}

		\caption{Typical application of the chasing strategy on the Petersen graph.}
		\label{fig:chasing_strategy}
	\end{figure}

	We now consider the cases of $\peter_5$ and $\peter_6$. In both cases,
	observe that $m$ and $m'$ are completely indistinguishable: $N(m)=N(m')$. It
	is then easily seen that the strategy for 2 cops on $\peter_5$ or $\peter_6$
	will be the same as the strategy developed above for $\peter_0$, except that
	the robber may choose to go to either $m$ or $m'$. We apply the strategy for
	$\peter_0$ by considering the robber to be on $m'$ whenever it is actually
	on $m$. This is essentially a simplified version of the well-known argument
	used to prove, in particular, Theorem \ref{generalretract}.

	Finally, we consider the cases $\peter_i$, $i\in \{1,\dots,4\}$. Our goal is
	to apply the strategy of $\peter_0$ developed above, with only slight
	modifications. Using that strategy, we choose initial positions for the cops
	in $\peter_i-m$ such that one of the cops is in $N[m']$ (it is described
	above why this is possible). If the robber chooses $m$ as an initial
	position, this cop may then move to $m'$. As $m'$ corners $m$, the robber
	cannot move without being captured. The other cop may then, within a few
	turns, capture the robber. Thus, the robber will choose an initial vertex in
	$\peter_i-m$. Now, as long as the robber is not on $m$, copy the strategy
	for $\peter_0$. Suppose that, at some point, the robber moves to $m$. In the
	strategy for $\peter_0$, the robber is adjacent to a cop before every one of
	its turns. Thus, this cop can move to a vertex adjacent to $m$. One easily
	verifies that in all graphs $\peter_i$ for $i\in \{1,\dots,4\}$, if one cop
	is adjacent to $m$, there is at most one other escape route $t$ for the
	robber. As $\peter_i-m\simeq \peter_0$ has diameter 2, the other cop can
	move to block this escape route by moving to some vertex in $N[t]$. Thus,
	while applying this strategy, the robber will never move to $m$. Hence, the
	strategy copied from $\peter_0$ yields the desired final position.
\end{proof}

By weakening the condition that it is the cops' turn at the end of the strategy,
we can get more freedom as to where we can place the cops, enabling more
strategies.

\begin{lemma}\label{chasinglemma3} Let $0\leq i \leq 6$. If $x,y,z$ are any
	three distinct vertices of $\peter_i-m$, then there exists a strategy for 2
	cops on $\peter_i$ to reach the following situation.
	\begin{enumerate}
		\item The robber is on $x$, except in the case $x=m'$ and $i\in\{5,6\}$,
		      where the robber is either on $m'$ or $m$.
		\item The cops are on $y$ and $z$.
		\item It is the robber's turn.
	\end{enumerate}
\end{lemma}
\begin{proof}
	Without loss of generality, we show the statement for $\peter_0$. For this
	lemma, generalizing to the cornered Petersen graphs is immediate.

	We first consider the case where $xz \in E(\peter_0)$. We will enumerate the
	main cases and conclude by symmetry for the others. We may assume that
	$x=\alpha_1$ and $z=\beta_1$ (using the labelling from Figure
	\ref{fig:petersenrobertson_graph}), all other possibilities can be solved
	using the automorphisms of Lemma \ref{petersenautomorphisms} (b). We apply
	Lemma \ref{chasinglemma1} to place the robber on vertex $\alpha_1$ and the
	cops on the vertices specified in Table \ref{table:strategy1} (always
	forming a strong stable set), and then specify the additional move required
	to place the cops in the desired final position.

	\begin{table}[h]
		\begin{tabular}{|c|c|c|}
			\hline
			\rowcolor{gray!60}
			Final position for cops ($y,z$) & Position after applying Lemma
			\ref{chasinglemma1}             & Movements \\ \hline
			$\alpha_2,\beta_1$              & $\beta_2, \beta_3$            &
			$\beta_2\rightarrow \alpha_2, \beta_3\rightarrow\beta_1$ \\ \hline
			$\beta_2,\beta_1$               & $\beta_2, \beta_3$            &
			$\beta_2\rightarrow\beta_2,  \beta_3\rightarrow\beta_1$  \\ \hline
			$\beta_3,\beta_1$               & $\beta_4, \beta_5$            &
			$\beta_4\rightarrow\beta_1, \beta_5\rightarrow\beta_3$   \\ \hline
			$\alpha_3,\beta_1$              & $\alpha_3, \beta_4$           &
			$\alpha_3\rightarrow\alpha_3, \beta_4\rightarrow\beta_1$ \\ \hline
		\end{tabular}
		\caption{Strategy on $\peter_0$ to bring the robber to $\alpha_1$ with the cops in the desired final position, when at least one cop needs to be adjacent to robber.}
		\label{table:strategy1}
	\end{table}

	It is easily seen that all other choices of $y$ are analogous by reflection
	of the graph relative to the vertical axis.

	We use a similar approach for the case where $xz \notin E$. We may suppose
	without loss of generality that $x=\alpha_1$ and $z=\beta_2$: it is easily
	verified that any two non-adjacent vertices can be expanded into a strong
	stable set, then apply Lemma \ref{petersenautomorphisms} (a). We must still
	verify the position can be reached for the different possible $y$. To
	further reduce the number of cases, we can also assume that $xy\notin E$ (if
	$xy\in E$, switching the roles of $y$ and $z$ brings us back to the previous
	case), as we can see in Table \ref{table:strategy2}.

	\begin{table}[h]
		\begin{tabular}{|c|c|c|}
			\hline
			\rowcolor{gray!60}
			Final position for cops ($y,z$) & Position after applying Lemma
			\ref{chasinglemma1}             & Movements \\ \hline
			$\alpha_3,\beta_2$              & $\alpha_3, \beta_4$           &
			$\alpha_3 \rightarrow\alpha_3, \beta_4 \rightarrow\beta_2$ \\ \hline
			$\alpha_4,\beta_2$              & $\alpha_3, \beta_4$           &
			$\alpha_3 \rightarrow\alpha_4, \beta_4 \rightarrow\beta_2$ \\ \hline
			$\beta_3,\beta_2$               & $\beta_3,\beta_2$             & $
			\beta_3 \rightarrow\beta_3, \beta_2 \rightarrow\beta_2$  \\ \hline
			$\beta_4,\beta_2$               & $\beta_4, \beta_5$            &
			$\beta_4 \rightarrow\beta_4, \beta_5 \rightarrow\beta_2$   \\ \hline
			$\beta_5,\beta_2$               & $\beta_4, \beta_5$            &
			$\beta_4 \rightarrow\beta_2, \beta_5 \rightarrow\beta_5$   \\ \hline
		\end{tabular}

		\caption{Strategy on $\peter_0$ to bring the robber to $\alpha_1$ with the cops in the desired final position, when neither cop needs to be adjacent to robber.}
		\label{table:strategy2}
	\end{table}
\end{proof}

In the next lemmas, we will consider graphs with the following properties, with
the goal of eventually showing that they do not exist. We state these properties
now to avoid repetition.

\begin{hypo}\label{labelhypo} Let $G$ be a connected graph such that $c(G)>3$
	and $u\in V(G)$ such that $G-N[u]\simeq \peter_i$, for some $0 \leq i \leq
	6$.
\end{hypo}

In the cases of $1\leq i\leq 6$, we may in particular consider that $m,m'\in
	V(G)$ by fixing the isomorphism. In the cases of $i=5,6$, as the labels $m$
	and $m'$ can be switched, we will always suppose $m'$ to be the vertex of
	the two which has the greatest degree in $G$ (if both have the same degree,
	then we choose arbitrarily).  To simplify notation, we will denote
	$B_u=V(G-N[u]-m)$. It is easily seen that in all cases $\langle B_u\rangle
	\simeq \peter_0$, that is $B_u$ induces a Petersen graph.

The approach will be to build up a number of structural properties of $G$ by
showing that otherwise there exists a winning strategy for $3$ cops, yielding a
contradiction. We start by proving that all vertices in $B_u$ have a neighbour
in $N(u)$.

\begin{lemma}\label{mindeg4} Consider Hypothesis \ref{labelhypo}. For all $x\in
		B_u$, $|N(x)\cap N(u)|\geq 1$.
\end{lemma}
\begin{proof}
	Let $x\in B_u$. We prove a restricted version of the statement before
	proving it generally: we first suppose that there exists a neighbour $v\neq
	m'$ of $x$ in $B_u$ such that $|N(v)\cap N(u)|\geq 1$. In the case of
	$\mathcal P_0$, then $m'$ does not exist so $v$ could be any vertex of
	$B_u$. Suppose that $|N(x)\cap N(u)| = 0$, with the goal of reaching a
	contradiction. We note that in the case with $x=m'$, then by our choice of
	$m'$ we also know that $m$ also has no neighbours in $N(u)$. We show this
	situation yields a winning strategy for 3 cops.

	Let $y,z$ be the other neighbours of $x$ in $B_u$ and let $w\in N(v)\cap
	N(u)$. The situation is portrayed in Figure \ref{fig:mindeg4}. We start by
	placing a cop on $u$, which will only move if the robber enters $N[u]$. This cop
	is commonly referred to as a stationary cop. As long as this cop stays on
	$u$, the robber is stuck in $G-N[u]\simeq \peter_i$. Thus, the two other
	cops may apply the strategy from Lemma \ref{chasinglemma3} on $G-N[u]$ to
	place the robber on $x$, a cop on $y$ and a cop on $z$. In the special case
	where $i \in \{5,6\}$ and $x=m'$, the robber might actually be on $m$.

	\begin{figure}[h]
		\begin{tikzpicture}[scale=3, dot/.style = {circle, fill, minimum
						size=#1, inner sep=0pt, outer sep=0pt}, dot/.default =
						5pt  
			]
			\node [dot] (0) at (0, 1) [label={above left: $z$}]{}; \node [dot]
			(2) at (-0.951057, 0.309017)  {}; \node [dot] (8) at
			(-0.587785,-0.809017)  {}; \node [dot] (9) at (0.587785,-0.809017)
			[label={below right: $v$}] {}; \node [dot] (3) at
			(0.951057,0.309017)  [label={above right: $x$}] {};

			\node [dot] (1) at (0,0.5*1)   {}; \node [dot] (4) at
			(0.5*-0.951057, 0.5*0.309017)   {}; \node [dot] (6) at
			(0.5*-0.587785,0.5*-0.809017)  {}; \node [dot] (7) at
			(0.5*0.587785,0.5*-0.809017)  {}; \node [dot] (5) at (0.5*0.951057,
			0.5*0.309017) [label={above left: $y$}]  {};

			\node [dot] (10) at (1, 1) [label={above: $m$}]{};

			\node [dot] (11) at (3, 0) [label={above: $u$}]{}; \node [dot] (12)
			at (2, 0) [label={above: $w$}]{};

			\draw (2,0) ellipse (0.4 and 0.75);

			\draw (11) to (12); \draw (9) to (12);

			\draw (0) to (1); \draw (0) to (2); \draw (0) to (3);

			\draw (1) to (6); \draw (1) to (7);

			\draw (2) to (8); \draw (2) to (4);

			\draw (4) to (5); \draw (4) to (7);

			\draw (5) to (6);

			\draw (8) to (6); \draw (8) to (9);

			\draw (9) to (7);

			\draw (9) to (3); \draw (3) to (5);

			\draw (10) to (2); \draw (10) to (3);

			\node at (2,0.9) {$N(u)$};

		\end{tikzpicture}
		\caption{Example situation during the proof of Lemma \ref{mindeg4}.
			Unused or unknown vertices and edges are omitted.}
		\label{fig:mindeg4}
	\end{figure}
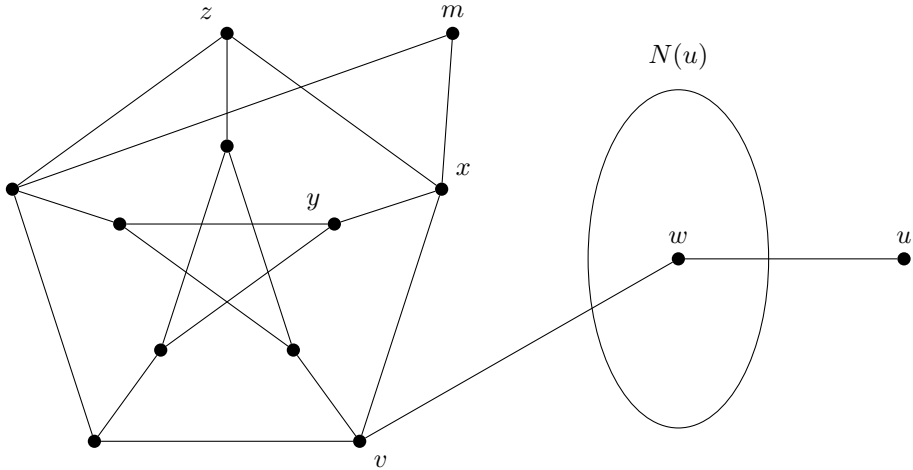

	During the last turn of this strategy, the cop on $u$ moves to $w$. It is
	now the robber's turn. All of the robber's neighbours in $B_u$ are protected
	by cops, there is a cop adjacent to the robber, and the robber has no
	neighbour in $N(u)$. Note that this statement is also true if the cop is on
	$m$, in the special case with $i=5,6$ and $x=m'$, as $m$ and $m'$ have the
	same neighbours $v,y,z,$ in $B_u$ (and both don't have neighbours in $N(u)$
	as stated above).

	The robber is caught unless it can move to an unprotected vertex outside
	$B_u$, which must necessarily be $m$. Note that there are some cases in which
	$m$ is still protected, for instance the special case discussed above. The
	robber was on $x$ and is now on $m$, so $x$ is adjacent to $m$. If one of
	the cops is adjacent to $m$, the robber is caught, so we assume otherwise,
	that $m$ is adjacent to neither $y$ nor $z$.

	If $x=m'$, the cop on $w$ moves back to $u$ and the cop on $y$ moves to
	$m'$: as $m'$ covers all neighbours of $m$ in $B_u$ and $u$ covers all
	neighbours in $N(u)$, the robber is trapped and will be caught one turn
	later.

	Now, suppose that $x \neq m'$. As $x$ is adjacent to $m$, $x$ is
	also adjacent to $m'$. Recall that we supposed that $v \neq m'$. Thus, $m'$
	is either $y$ or $z$. The cop on $m'$ stays put and the cop on $w$ moves
	back to $u$, trapping the robber on $m$. The last cop may then capture the
	robber within a few turns.

	In all cases, a contradiction is reached with the hypothesis that $c(G) >
	3$. This proves the specific case in which $x$ has a neighbour $v\in B_u$
	(respecting $v\neq m'$) which has at least one neighbour in $N(u)$. We now
	wish to prove the full statement without this condition.

	With what we have proved so far, we know that as soon as a vertex of $B_u$
	(other than $m'$) has a neighbour in $N(u)$, we can say the same for its 3
	neighbours in $B_u$ (and then we can repeat this argument to their
	neighbours, etc.). Thus, in order to prove the lemma, it suffices to show
	that there exists at least one vertex of $B_u\setminus \{m'\}$ that has a
	neighbour in $N(u)$ (because $\langle B_u\rangle-m'$ is necessarily
	connected).

	Suppose the contrary: that no vertex of $B_u\setminus \{m'\}$ has a
	neighbour in $N(u)$. If we are in the case of $G-N[u]\simeq \peter_0$, then
	$G$ would be disconnected (as $B_u\setminus \{m'\}=B_u$), which is a
	contradiction. In the remaining cases, we give a winning strategy for 3
	cops. Place a stationary cop on a vertex $t\in N[m]\cap N[m']\cap B_u$,
	place one cop on $u$, and place the third cop anywhere. The robber must
	choose an initial position in $B_u\setminus N[t]$. As any exit from $B_u$
	will go through $m$ or $m'$ (by our hypothesis), the stationary cop on $t$
	guarantees that the robber will never leave $B_u$. The two other cops then
	have a winning strategy on the component of $\langle B_u\rangle- N[t]$
	containing the robber, which has at most 6 vertices.
\end{proof}

We will frequently use this idea of applying the chasing strategy on $G-N[u]$
while we leave a cop on $u$. As long as there is a cop on $u$, it is as if we
were playing on $G-N[u]$. Moving the cop from $u$ during the last move of this
strategy does not affect it, as it happens after the last robber move which is
part of that strategy.

Similarly to in the first part of the last proof, many of the arguments in this
section will end with the only remaining possible situation being that there is
a cop on $u$, a cop on $m'$, and the robber on $m$. In this situation, the
robber cannot move, as $m'$ corners $m$ in $G-N[u]$, and a third cop may, within
a few turns, capture the robber.

We now characterize the neighbourhoods in $B_u$ of vertices in $N(u)$.

\begin{lemma}\label{neighbourhoodproperty} Consider Hypothesis \ref{labelhypo}.
	If $w \in N(u)$, then $N(w) \cap B_u$ does not contain a subset $\{a,b,c\}$
	of distinct vertices such that :
	\begin{enumerate}
		\item $ab \notin E(G)$;
		\item $c \notin N(a) \cap N(b)$ ($c$ is not the common neighbour of $a$
		      and $b$ in $\langle B_u\rangle $);
		\item $c \notin N(x)$ where $\{x\}=N(a) \cap N(b)\cap B_u$ ($c$ is not
		      adjacent to the common neighbour of $a$ and $b$ in $\langle
		      B_u\rangle$).
	\end{enumerate}
\end{lemma}
\begin{proof}
	Suppose that $N(w)\cap B_u$ does contains a subset $\set{a,b,c}$ respecting
	these conditions. We give a winning strategy for 3 cops on $G$, which
	will lead to a contradiction. We denote by $x$ the common neighbour of $a$
	and $b$ in $B_u$, which exists by condition $(1)$ as vertices of distance 2
	in a Petersen graph have exactly one common neighbour. Furthermore by
	condition $(2)$, $c\neq x$. Denote by $d$ the neighbour of $x$ in $B_u$ that
	is neither $a$ or $b$. Then by condition $(3)$, $c\neq d$.

	Let $z$ be a vertex of $B_u$ such that $\{x,c,z\}$ is a strong stable set of
	$\langle B_u\rangle$ (it is easily seen that any stable set of size 2 in the
	Petersen graph can be expanded into a strong stable set). The situation is
	portrayed in Figure \ref{fig:neighbourhoodproperty}. We place a cop on $u$
	at the start of the game, and then use Lemma \ref{chasinglemma1} to place
	the other cops on $c$ and $z$, and the robber on $x$ (or if $x=m'$ and $i\in
	\{5,6\}$ possibly on $m$). As stated in Lemma \ref{chasinglemma1}, it is now
	the cops' turn.

	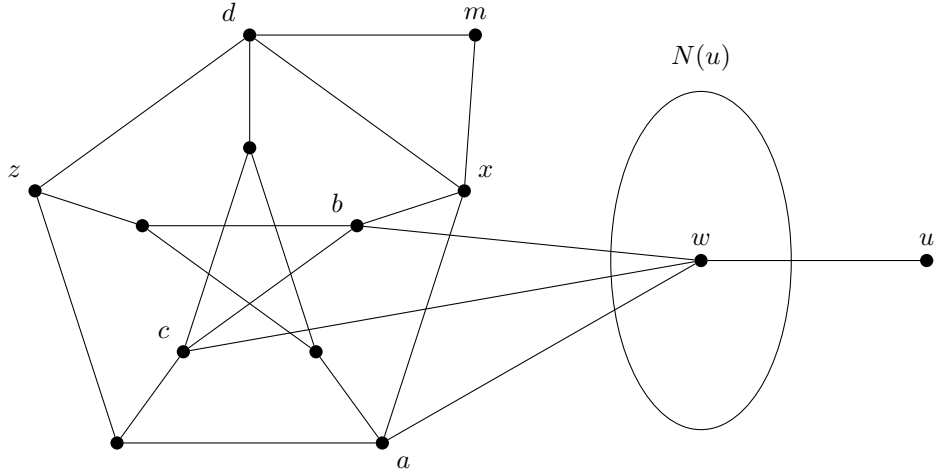
\begin{figure}[h]
		\begin{tikzpicture}[scale=3, dot/.style = {circle, fill, minimum
						size=#1, inner sep=0pt, outer sep=0pt}, dot/.default =
						5pt  
			]
			\node [dot] (0) at (0, 1) [label={above left: $d$}]{}; \node [dot]
			(2) at (-0.951057, 0.309017) [label={above left: $z$}] {}; \node
			[dot] (8) at (-0.587785,-0.809017)  {}; \node [dot] (9) at
			(0.587785,-0.809017) [label={below right: $a$}] {}; \node [dot] (3)
			at (0.951057,0.309017)  [label={above right: $x$}] {};

			\node [dot] (1) at (0,0.5*1)  {}; \node [dot] (4) at (0.5*-0.951057,
			0.5*0.309017)  {}; \node [dot] (6) at (0.5*-0.587785,0.5*-0.809017)
			[label={above left: $c$}] {}; \node [dot] (7) at
			(0.5*0.587785,0.5*-0.809017) {}; \node [dot] (5) at (0.5*0.951057,
			0.5*0.309017) [label={above left: $b$}] {};

			\node [dot] (10) at (1, 1) [label={above: $m$}]{};

			\node [dot] (11) at (3, 0) [label={above: $u$}]{}; \node [dot] (12)
			at (2, 0) [label={above: $w$}]{};

			\draw (2,0) ellipse (0.4 and 0.75);

			\draw (6) to (12); \draw (5) to (12); \draw (11) to (12); \draw (9)
			to (12);

			\draw (0) to (1); \draw (0) to (2); \draw (0) to (3);

			\draw (1) to (6); \draw (1) to (7);

			\draw (2) to (8); \draw (2) to (4);

			\draw (4) to (5); \draw (4) to (7);

			\draw (5) to (6);

			\draw (8) to (6); \draw (8) to (9);

			\draw (9) to (7);

			\draw (9) to (3); \draw (3) to (5);

			\draw (10) to (0); \draw (10) to (3);

			\node at (2,0.9) {$N(u)$};

		\end{tikzpicture}
		\caption{Example situation during the proof of Lemma
			\ref{neighbourhoodproperty}. Unused or unknown vertices and edges
			are omitted.}
		\label{fig:neighbourhoodproperty}
	\end{figure}

	The cop on $c$ moves to $w$, the cop on $z$ moves to either $d$ or a
	neighbour of $d$ (this is possible because the Petersen graph has diameter
	2), and the cop on $u$ stays still. All neighbours of $x$ in $N(u)$ are
	covered by the cop on $u$, the vertices $a$ and $b$ are covered by the cop
	on $w$, and $d$ is covered by the 3rd cop, which is either on $d$ or on a
	neighbour of $d$. Hence, after the robber's move the robber must either be
	on $x$ or on $m$ (if $m$ exists, that is $1\leq i\leq 6$, and if moving to
	$m$ from $x$ is possible). We note that in the special case with $x=m'$ and
	$i\in \{5,6\}$ the robber might have already been on $m$, but then the same
	argument as above shows that the robber could not have moved outside of
	either $x=m'$ or $m$.

	If the robber is still on $x$, but cannot be immediately captured, the cop
	which is adjacent to $d$ moves to $d$. Now, the robber cannot stay put
	without being captured. Hence, for the rest of the proof we can assume the
	robber has now moved to $m$ or was already on $m$ (in the latter case, there
	is not necessarily a cop now on $d$, as we only do this move if the robber
	was still on $x$). This in particular excludes the case $i=0$.

	If $m'$ is $a$ or $b$, the cop on $w$ moves to $m'$. If $m'=d$, the cop that
	is adjacent to $d$ or on $d$ moves to (or stays on) $d$. In both cases,
	there is now a cop on $m'$, which, together with the cop on $u$, guarantees
	that the robber is now stuck on $m$. The third cop may capture the robber
	within a few turns.

	If $m'=x$, then, by definition, $N(m)\cap B_u\subseteq \{a,b,d,x\}$. As
	previously done, move a cop to $d$ (if it is not already there). The pair of cops
	on $d$ and $w$ cover this set, hence the robber cannot move. At the next
	cops' turn, the cop on $d$ moves to $m'$. Then, we are in the same situation
	as above.

	In all cases, there is a contradiction as $c(G)>3$.
\end{proof}

Although the statement of the previous lemma appears somewhat convoluted, it can
essentially be reformulated as the following lemma.

\begin{lemma}\label{dominatingneighbourhood} Consider Hypothesis
	\ref{labelhypo}. If $w \in N(u)$, then there exists a vertex of $B_u$
	dominating $N(w) \cap B_u$.
\end{lemma}
\begin{proof}
	Recall that $\langle B_u\rangle$ is a Petersen graph. The result is trivial
		if $|N(w) \cap B_u|\leq 2$ as the diameter of $\langle B_u$ is 2. If
		$|N(w)\cap B_u|\geq 3$, suppose the statement is false. As $\langle
		B_u\rangle$ does not contain a triangle, not all vertices of $N(w)\cap
		B_u$ can be pairwise adjacent, so we can choose $a,b\in N(w)\cap B_u$
		such that $a,b$ are not adjacent. Denote $x$ the common neighbour of
		$a,b$ in $B_u$. By our supposition, $x$ does not dominate $N(w)\cap
		B_u$, thus we can choose $c\in N(w)\cap B_u$ not in $N[x]$. The subset
		$\{a,b,c\}$ then contradicts Lemma \ref{neighbourhoodproperty}.
\end{proof}

In particular, every vertex of $N(u)$ can have at most 4 neighbours in $B_u$
because $\langle B_u \rangle$ is 3-regular. We also note that in some of the
cases there is a unique choice for this dominating vertex, in particular when
$N(w)\cap B_u$ has 3 or 4 vertices. We are now ready to strengthen Lemma
\ref{mindeg4}.

\begin{lemma}\label{mindeg6} Consider Hypothesis \ref{labelhypo}. For all $x\in
		B_u$, the following holds.
	\begin{enumerate}
		\item If $x\notin N[m']$, then $|N(x)\cap N(u)|\geq 3$.
		\item If $x\in N[m']$, then $|N(x)\cap N(u)|\geq 2$.
	\end{enumerate}
\end{lemma}
\begin{proof}\item
	\begin{enumerate}
		\item Suppose the contrary: there exists $x\in B_u\setminus N[m']$ such
		      that $|N(x)\cap N(u)| \in \{1,2\}$ (as by Lemma \ref{mindeg4},
		      $|N(x)\cap N(u)|\geq 1$). We give a winning strategy for 3
		      cops on $G$. Denote by $w_1,w_2$ the neighbours of $x$ in $N(u)$
		      (if there is only one neighbour, set $w_1=w_2$) and by
		      $y_1,y_2,y_3$ the neighbours of $x$ in $B_u$.

		      By Lemma \ref{dominatingneighbourhood}, there exists a vertex of
		      $B_u$ dominating the neighbourhood of $w_1$ in $B_u$. As $x$ is in
		      this neighbourhood, we know this dominating vertex (there might be
		      more than one possible choice) is in $\{y_1,y_2, y_3,x\}$. This is
		      also true for $w_2$. Thus, we can pick at most 2 elements of
		      $\{y_1,y_2,y_3,x\}$ that dominate all neighbours of $w_1,w_2$ in
		      $B_u$.

		      Without loss of generality (by symmetry of $y_1,y_2,y_3$ in the
		      Petersen graph), we assume the vertices dominating the
		      neighbourhoods of $w_1$ and of $w_2$ in $B_u$ are in
		      $\{y_1,y_2,x\}$. By Lemma \ref{mindeg4}, $y_3$ must have a
		      neighbour $t$ in $N(u)$. The situation is portrayed in Figure
		      \ref{fig:mindeg6}.

		      \begin{figure}[h]
			      \begin{tikzpicture}[scale=3, dot/.style = {circle, fill,
							      minimum size=#1, inner sep=0pt, outer
							      sep=0pt}, dot/.default = 5pt
				      ]
				      \node [dot] (0) at (0, 1) [label={above left: $y_3$}]{};
				      \node [dot] (2) at (-0.951057, 0.309017)  {}; \node [dot]
				      (8) at (-0.587785,-0.809017)  {}; \node [dot] (9) at
				      (0.587785,-0.809017) [label={below right: $y_1$}] {};
				      \node [dot] (3) at (0.951057,0.309017)  [label={above
				      right: $x$}] {};

				      \node [dot] (1) at (0,0.5*1)  {}; \node [dot] (4) at
				      (0.5*-0.951057, 0.5*0.309017)  {}; \node [dot] (6) at
				      (0.5*-0.587785,0.5*-0.809017) {}; \node [dot] (7) at
				      (0.5*0.587785,0.5*-0.809017)  {}; \node [dot] (5) at
				      (0.5*0.951057, 0.5*0.309017) [label={above left: $y_2$}]
				      {};

				      \node [dot] (10) at (1, 1) [label={above: $m$}]{};

				      \node [dot] (11) at (3, 0) [label={above: $u$}]{}; \node
				      [dot] (12) at (2, -0.4) [label={above: $w_1$}]{}; \node
				      [dot] (13) at (2, 0) [label={above: $w_2$}]{}; \node [dot]
				      (14) at (2, 0.4) [label={above: $t$}]{};

				      \draw (2,0) ellipse (0.4 and 0.75);

				      \draw (8) to (12); \draw (9) to (12);

				      \draw (5) to (13);

				      \draw (0) to (14); \draw (11) to (12); \draw (11) to (13);
				      \draw (11) to (14);

				      \draw (3) to (13); \draw (3) to (12);

				      \draw (0) to (1); \draw (0) to (2); \draw (0) to (3);

				      \draw (1) to (6); \draw (1) to (7);

				      \draw (2) to (8); \draw (2) to (4);

				      \draw (4) to (5); \draw (4) to (7);

				      \draw (5) to (6);

				      \draw (8) to (6); \draw (8) to (9);

				      \draw (9) to (7);

				      \draw (9) to (3); \draw (3) to (5);

				      \draw (10) to (2); \draw (10) to (13);

				      \node at (2,0.9) {$N(u)$};

			      \end{tikzpicture}
			      \caption{Example situation during the proof of Lemma
				      \ref{mindeg6} (1). Unused or unknown vertices and edges
				      are omitted.}
			      \label{fig:mindeg6}
		      \end{figure}
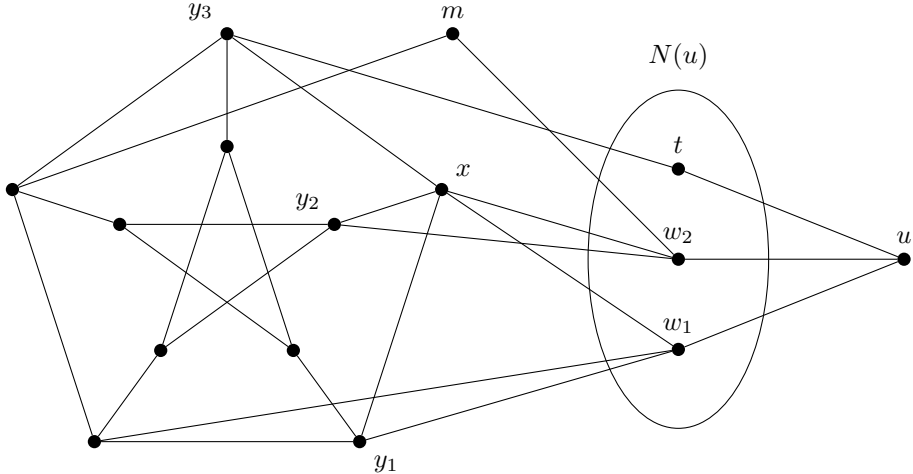

		      We place one cop on $u$. We use Lemma \ref{chasinglemma3} to place
		      the robber on $x$ and the two other cops on $y_1, y_2$. During the
		      last move of this strategy, the cop on $u$ moves to $t$. It is now
		      the robber's turn. The robber on $x$ cannot move to a neighbour
		      inside of $B_u$ (there are cops on $y_1$ and $y_2$, and $y_3$ is
		      covered by the cop on $t$) and there are cops adjacent to the
		      robber. As $x \notin N[m']$, we know that $m \notin N(x)$. Thus,
		      the robber has no choice but to move to either $w_1$ or $w_2$. Let
		      us say the robber moves to $w_1$ (the strategy for $w_2$ is
		      analogous).

		      Denote by $a$ the vertex dominating the neighbours of $w_1$ in
		      $B_u$. We recall that $a$ is either $y_1$, $y_2$ or $x$. We now
		      move the cop on $t$ back to $u$. Of the two cops on $y_1$ and
		      $y_2$, one must be able to move to $a$, and does so. If $m'\in
		      B_u$ (that is, if we are not in the case of $\peter_0$), the third
		      cop moves to either $m'$ or a neighbour of $m'$. After this move,
		      all escapes in $N(u)$ are covered by the cop on $u$, all escapes
		      in $B_u$ are covered by the cop on $a$, and the robber cannot stay
		      still as $u$ is adjacent to $w_1$. Thus, the robber is caught one
		      move later unless the robber can move to $m$. In this case, the
		      third cop can now move to $m'$ and trap the robber. Leaving the
		      cops on $u$ and $m'$ fixed, the cop on $a$ can then go capture the
		      robber. This is a contradiction as $c(G) > 3$.

		\item As the proof will be very similar to the previous case, we outline
		      the main differences. We suppose the statement is false, so that
		      $m'$ has at most one neighbour in $N(u)$ (and so exactly 1, by
		      Lemma \ref{mindeg4}).

		      If $x=m'$ and $i \in \{5,6\}$, recall that $m'$ is chosen to be of
		      higher degree (or equal) than $m$, so $m'$ has exactly one
		      neighbour in $N(u)$, and $m$ has at most one neighbour in $N(u)$.
		      Then, we denote by $w_1$ and $w_2$ these neighbours and apply the
		      same strategy as above. Even though the robber will have chosen to
		      go to either $m'$ or $m$, both cases for its subsequent move will
		      be covered using the strategy above. In particular, we note that
		      the robber will not be able to stay on either $x=m'$ or $m$, as
		      both are adjacent to $y_1,y_2$, so the robber will indeed have to
		      move to either $w_1$ or $w_2$. Thus, either $m$ or $m'$ must have
		      2 or more neighbours in $N(u)$. As we have selected $m'$ to have
		      the greatest degree of the two, the statement follows for this
		      case.

		      Consider now that $x \in N[m']$ but $x \neq m'$ or $i \notin
		      \set{5,6}$ (as we covered that case above). Our goal is to prove
		      that $x$ cannot have a unique neighbour in $N(u)$. Since we are
		      supposing the contrary, let $w_1$ be a unique neighbour of $x$ in
		      $N(u)$. As in the above strategy, one cop's role will be to cover
		      the vertex dominating the neighbourhood of $w_1$ in $B_u$, or if
		      this is $x$, then to be on an adjacent vertex. We will also want a
		      cop to be on $m'$, or on a neighbour of $m'$ if $x=m'$. In the
		      notation of the original case, this could informally be seen as
		      considering $w_2$ to be $m$. If the robber moves to $w_1$, we
		      follow a similar strategy to above. If the robber moves to $m$,
		      then one of the cops will move to $m'$. Recalling that another
		      robber will return to $u$, the last cop will be able to go capture
		      the robber.
	\end{enumerate}
\end{proof}

We are now ready to prove the desired results.

\begin{proposition}\label{propn11}\label{propn12} If $G$ is a connected graph
	such that $\Delta\in \{n-12,n-11\}$ and $n \leq 18$, then $c(G)\leq 3$.
\end{proposition}
\begin{proof}\item
	\begin{enumerate}
		\item We consider $\Delta=n-11$. Let $u$ be a vertex of maximum degree.
		      We know that $|V(G-N[u])|=10$. If $G-N[u]$ is disconnected, each
		      one of its connected components has cop number at most 2, as no
		      connected component can contain at least 10 vertices. Applying
		      Corollary \ref{retractcomponents} yields the desired result.
		      Otherwise, $G-N[u]$ must be connected. Suppose that $c(G)>3$.
		      Then, $c(G-N[u])>2$, and by Theorem \ref{min3}, $G-N[u]$ must be
		      isomorphic to $\peter_0$. Then, $G$ and $u$ satisfy the conditions
		      of Hypothesis \ref{labelhypo}.

		      By Lemma \ref{mindeg6}, every vertex of $B_u$ has at least 3
		      neighbours in $N(u)$ (recall that $m'$ does not exist in $\mathcal
		      P_0$). As $B_u$ has 10 vertices, this means there are at least 30
		      edges between $N(u)$ and $B_u$.

		      As $n\leq 18$, we have that $\Delta\leq 7$. By Lemma
		      \ref{dominatingneighbourhood}, every vertex of $N(u)$ has at most
		      $4$ neighbours in $B_u$. Thus, there are at most $4\Delta\leq 28$
		      edges between $N(u)$ and $B_u$. This is a contradiction, as we
		      have claimed there are at least 30 but at most 28 edges between
		      $N(u)$ and $B_u$. Thus, $c(G)\leq 3$.

		\item We consider $\Delta=n-12$. Let $u$ be a vertex of maximum degree.
		      We know that $|V(G-N[u])|=11$. Let us first consider the case
		      where $G-N[u]$ is disconnected. If every connected component has
		      cop number at most 2, then, as in the previous case, we are done.
		      By Theorem \ref{min3}, the only other case is if one component is
		      isomorphic to $\peter_0$ and the other is an isolated vertex $x$.
		      By applying Corollary \ref{retractcomponents} (or Corollary
		      \ref{removecorner}), $c(G)\leq 3$ if and only if $c(G-x)\leq 3$.
		      As $G-x$ satisfies the conditions of the previous case of this
		      proposition, we conclude that $c(G-x)\leq 3$.

		      We may now consider that $G-N[u]$ is connected. Suppose $c(G)>3$.
		      By Proposition \ref{classification11}, $G-N[u]\simeq \peter_i$,
		      for some $1\leq i\leq 6$. Then, $G$ and $u$ satisfy the condition
		      of Hypothesis \ref{labelhypo}.

		      By Lemma \ref{mindeg6}, every vertex of $B_u$ has at least 2
		      neighbours in $N(u)$, and in fact those not in $N[m']$ (of which
		      there are at least 6) have at least 3 neighbours in $N(u)$. In
		      total, there are at least 26 edges between $N(u)$ and $B_u$.

		      By Lemma \ref{dominatingneighbourhood}, each vertex of $N(u)$ has
		      at most $4$ neighbours in $B_u$. As $n \leq 18$, we have that
		      $\Delta\leq 6$. Thus, there are at most $4\Delta\leq 24$ edges
		      between $N(u)$ and $B_u$. This is a contradiction, as we have
		      claimed there are at least 26 but at most 24 edges between $N(u)$
		      and $B_u$, hence $c(G)\leq 3$.
	\end{enumerate}
\end{proof}

These results will be used to prove that $M_4=19$ in Section \ref{mainsection}.
Readers only interested in the main result may skip the remainder of this
section, as it concerns reducing the number of possible 4-cop-win graphs on 19
vertices. We first need the following definition.

\begin{definition}\label{projectionclaw} Consider Hypothesis \ref{labelhypo}.
	Let $w\in N(u)$ such that $|B_u\cap N(w)|=4$. The vertex $x\in B_u$
	dominating $B_u\cap N(w)$ will be called the \emph{projection} of $w$. If
	$x$ is the projection of $k$ vertices of $N(u)$, we will call $p(x)=k$ the
	\emph{projection multiplicity} of $x$.
\end{definition}

By Lemma \ref{dominatingneighbourhood}, this is well defined and the projection
of a vertex is unique.

\begin{observation}
	\label{observation:on_projective_multiplicity}
	Consider Hypothesis \ref{labelhypo}. If $x\in B_u$, $| N(x) \cap N(u)|\geq
	\sum_{y\in N[x]\cap B(u)} p(y)$. In particular, if $y \in B_u$ has
	projection multiplicity $k$, then each vertex in $N(y) \cap B_u$ has at
	least $k$ neighbours in $N(u)$.
\end{observation}

\begin{proof}
	We recall that when a vertex $y$ has projection multiplicity $k$, this means
	that $k$ vertices of $N(u)$ have for neighbours in $B_u$ exactly $N[y]\cap
	B_u$, giving each vertex in this set at least $k$ neighbours in $N(u)$.

	Noting that the projection of a vertex is unique, we see that the neighbours
	that $x$ inherits from each projection on $x$ or on its neighbours in $B_u$
	are pairwise distinct. The lower bound follows immediately by summing the
	projection multiplicity for each vertex in $N[x]$.
\end{proof}

We now see an interesting property of projections.

\begin{lemma}\label{lemmaprojection} Consider Hypothesis \ref{labelhypo}. Let
	$x\in B_u\setminus N[m']$.
	\begin{enumerate}
		\item If $|N(x)\cap N(u)|=3$, then $p(x)=0$.
		\item More generally, $p(x)\leq |N(x)\cap N(u)|-2$.
	\end{enumerate}
\end{lemma}
\begin{proof}\item
	\begin{enumerate}
		\item Supposing the contrary, we give a winning strategy for 3 cops.
		      Let us say that $x$ is the projection of a vertex $w$ of $N(u)$:
		      $w$ is adjacent to $x$ and to each neighbour of $x$ in $B_u$. As
		      $|N(x)\cap N(u)|=3$, $x$ has two other neighbours in $N(u)$, which
		      we will denote by $t_1$ and $t_2$. If $t_1$ has a neighbour in
		      $B_u$ other than $x$, choose one and denote it $r_1$. If not, then
		      choose $r_1$ to be any neighbour of $x$ in $B_u$. We choose $r_2$
		      similarly. The situation is portrayed in Figure
		      \ref{fig:lemmaprojection}.

		      \begin{figure}[h]
			      \begin{tikzpicture}[scale=3, dot/.style = {circle, fill,
							      minimum size=#1, inner sep=0pt, outer
							      sep=0pt}, dot/.default = 5pt
				      ]
				      \node [dot] (0) at (0, 1) [label={above left: $r_2$}]{};
				      \node [dot] (2) at (-0.951057, 0.309017)  {}; \node [dot]
				      (8) at (-0.587785,-0.809017)  {}; \node [dot] (9) at
				      (0.587785,-0.809017) [label={below right: $r_1$}] {};
				      \node [dot] (3) at (0.951057,0.309017)  [label={above
				      right: $x$}] {};

				      \node [dot] (1) at (0,0.5*1)  {}; \node [dot] (4) at
				      (0.5*-0.951057, 0.5*0.309017)  {}; \node [dot] (6) at
				      (0.5*-0.587785,0.5*-0.809017) {}; \node [dot] (7) at
				      (0.5*0.587785,0.5*-0.809017)  {}; \node [dot] (5) at
				      (0.5*0.951057, 0.5*0.309017) {};

				      \node [dot] (11) at (3, 0) [label={above: $u$}]{}; \node
				      [dot] (12) at (2, -0.4) [label={above: $t_1$}]{}; \node
				      [dot] (13) at (2, 0) [label={above: $w$}]{}; \node [dot]
				      (14) at (2, 0.4) [label={above: $t_2$}]{};

				      \draw (2,0) ellipse (0.4 and 0.75);

				      \draw (8) to (12); \draw (9) to (12);

				      \draw (5) to (13);

				      \draw (3) to (14);

				      \draw (0) to (14); \draw (11) to (12);

				      \draw (11) to (13); \draw (11) to (14);

				      \draw (9) to (13); \draw (0) to (13);

				      \draw (3) to (13); \draw (3) to (12);

				      \draw (0) to (1); \draw (0) to (2); \draw (0) to (3);

				      \draw (1) to (6); \draw (1) to (7);

				      \draw (2) to (8); \draw (2) to (4);

				      \draw (4) to (5); \draw (4) to (7);

				      \draw (5) to (6);

				      \draw (8) to (6); \draw (8) to (9);

				      \draw (9) to (7);

				      \draw (9) to (3); \draw (3) to (5);

				      \node at (2,0.9) {$N(u)$};

			      \end{tikzpicture}
			      \caption{Example situation during the proof of Lemma
				      \ref{lemmaprojection} (1). Unused or unknown vertices and
				      edges are omitted.}
			      \label{fig:lemmaprojection}
		      \end{figure}
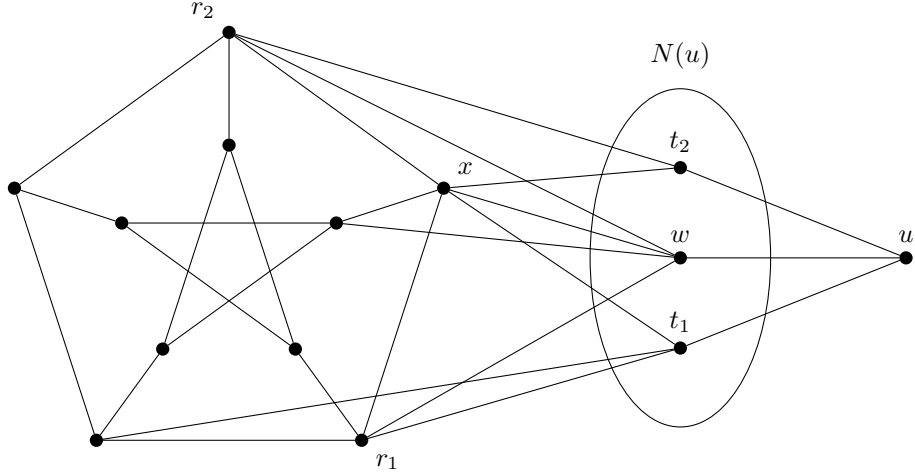
		      We start by placing one cop on $u$. Using Lemma
		      \ref{chasinglemma3} (recall that $x\neq m'$, which avoids the
		      exceptional case), we place the robber on $x$ and the two other
		      cops on $r_1$ and $r_2$ (if $r_1=r_2$, then place one cop there
		      and another cop on any vertex in $B_u\setminus\{x,r_1\}$). During
		      the last move of that strategy, move the cop from $u$ to $w$. It
		      is now the robber's turn. As there is a cop on $w$, the robber
		      cannot stay in $B_u$. As $x\notin N[m']$, $x$ is not adjacent to
		      $m$. If $t_1$ had a neighbour in $B_u$ other than $x$, then the
		      cop on $r_1$ blocks the robber from moving to $t_1$, and similarly
		      for $t_2$.

		      Thus, the only scenario in which the robber does not get captured
		      immediately after moving is if (without loss of generality), $t_1$
		      only has one neighbour in $B_u$, and the robber moves to $t_1$. In
		      this case, we chose $r_1$ to be some neighbour of $x$. The cop on
		      $w$ moves back to to $u$, and the cop on $r_1$ moves to $x$. The
		      third cop moves to $m'$ or a neighbour of $m'$ (if $m'$ is in the
		      graph). The robber is caught one turn later, as $x$ is the only
		      neighbour of $t_1$ in $B_u$ and $u$ dominates $N(u)$, unless the
		      robber can move to $m$. In this case, the third cop can move to
		      $m'$ and trap the robber. The cop on $x$ can capture the robber
		      within a few turns. This contradicts that $c(G)>3$.

		\item The strategy is similar to the previous case. Suppose, to the
		      contrary, that $x$ has projection multiplicity at least $|N(x)\cap
		      N(u)|-1$. Then, there is at most 1 neighbour of $x$ in $N(u)$
		      which does not project onto $x$. Choose $t_1$ to be this vertex
		      (if there is any) and select the corresponding $r_1$ as above.
		      Choose $r_2$ to be any other neighbour of $x$ in $B_u$: $r_2$
		      covers all vertices projecting onto $x$. The rest of the strategy
		      is identical.
	\end{enumerate}
\end{proof}

We are now ready for the desired result.

\begin{proposition}\label{propositionn19D78} If $G$ is a connected graph such
	that $n=19$ and $\Delta\in \{7,8\}$, then $c(G)\leq 3$.
\end{proposition}

\begin{proof}
	Suppose $c(G)>3$. Let $u$ be a vertex of maximal degree in $G$.
	\begin{enumerate}

		\item We consider $\Delta=8$. Recall the arguments of the proof of
		      Proposition \ref{propn11}. In particular, we can consider that
		      $G-N[u]\simeq \peter_0$.

		      There are at most $4\Delta=32$ edges between $N(u)$ and $B_u$, by
		      Lemma \ref{dominatingneighbourhood}. By Lemma \ref{mindeg6}, each
		      vertex in $B_u$ has at least 3 neighbours in $N(u)$: there are at
		      least 30 edges between $B_u$ and $N(u)$. Thus, there are at most 2
		      extra edges. By extra edges, we mean that these are edges which,
		      if removed, would leave each vertex in $B_u$ with exactly the
		      lower bound number of neighbours in $B_u$, as specified in Lemma
		      \ref{mindeg6}. Then, there are at least 8 vertices in $B_u$ which
		      have no extra edges, having exactly 3 neighbours in $N(u)$.

		      Furthermore, if there are fewer than $6$ vertices of $N(u)$ that
		      each have exactly 4 neighbours in $B_u$, then there cannot be at
		      least 30 edges between $N(u)$ and $B_u$. Thus, $\sum_{x\in B_u}
		      p(x)\geq 6$.

		      Recall that Lemma \ref{lemmaprojection} states that no vertex in
		      $N(u)$ with exactly 3 neighbours in $B_u=B_u\setminus N[m']$ can
		      be a projection. Hence, only the vertices with extra edges may
		      have non-zero projection multiplicity. The first consequence of
		      this is that it is impossible for all vertices of $B_u$ to have
		      exactly 3 neighbours in $N(u)$. Furthermore, we have seen that
		      there are at most 2 vertices which have extra edges (these are the
		      vertices of $B_u$ that have either $4$ or $5$ neighbours in
		      $N(u)$). Denote them by $a_1,a_2$ (if there is only one vertex,
		      let $a_1=a_2$). Then, $p(a_1)+p(a_2)\geq 6$ (if $a_1=a_2$, then simply
		      say $p(a_1)\geq 6$). Let $x\in N[a_1]\cap N[a_2]\cap B_u$ (which
		      exists as $\peter_0$ has diameter 2). As $x$ is adjacent to all
		      projections, $x$ must be adjacent at least 6 vertices of $N(u)$
		      (Observation \ref{observation:on_projective_multiplicity}). As $x$
		      also has 3 neighbours in $B_u$, the degree of $x$ is at least $9$,
		      which is a contradiction.

		\item We consider $\Delta=7$. Recall the arguments of the proof of
		      Proposition \ref{propn11}. In particular, we can say that
		      $G-N[u]\simeq \peter_i$, for some $1\leq i \leq 6$.

		      There are at most $4\Delta=28$ edges between $N(u)$ and $B_u$, by
		      Lemma \ref{dominatingneighbourhood}. By Lemma \ref{mindeg6}, each
		      vertex in $B_u\setminus N[m']$ has at least 3 neighbours in $N(u)$
		      and each vertex in $B_u\cap N[m']$ has at least 2 neighbours in
		      $N(u)$: in total, there are at least 26 edges between $B_u$ and
		      $N(u)$. Using the same argument as above, depending on the number
		      of edges between $B_u$ and $N(u)$, we can find between 5 and 7
		      vertices in $N(u)$ which have 4 neighbours each in $B_u$, and thus
		      the total projection multiplicity of $B_u$ is as follows.
		      \begin{enumerate}
			      \item 26 edges: $\sum_{x\in B_u} p(x)\geq 5$;
			      \item 27 edges: $\sum_{x\in B_u} p(x)\geq 6$;
			      \item 28 edges: $\sum_{x\in B_u} p(x)= 7$ (as there are
			      exactly $7$ vertices in $N(u)$ there cannot be more than $7$
			      projections).
		      \end{enumerate}

		      Recall that Lemma \ref{lemmaprojection} states that no vertex in
			      $B_u \setminus N[m']$ with 3 neighbours in $N(u)$ can be a
			      projection. Also, if $x \in B_u \setminus N[m']$, $p(x) \leq
			      |N(x) \cap N(u)|-2$: if $x$ has 4 neighbours in $N(u)$ it can
			      be the projection of at most 2 vertices.

		      If all vertices in $B_u \setminus N[m']$ have exactly 3 neighbours
		      in $N(u)$, then this implies all projections will be vertices in
		      $N[m']$: at least 5 vertices project on $m'$ or on a neighbour.
		      Thus, $m'$ will have at least 5 neighbours in $N(u)$. As $m'$ also
		      has at least 3 neighbours in $B_u$, $d(m')\geq 8$, which is
		      impossible as $\Delta=7$. This situation includes the case in
		      which there are exactly 26 edges between $B_u$ and $N(u)$.

		      Suppose there is exactly one vertex $x$ of $B_u\setminus N[m']$
		      with exactly 4 neighbours in $N(u)$, with all others having
		      exactly 3. This vertex will have projection multiplicity at most
		      2, so the total projection multiplicity of vertices of $N[m']$ is
		      at least 4. Thus, $m'$ will have at least 4 neighbours in $N(u)$,
		      which is impossible as this would imply there are 29 edges between
		      $B_u$ and $N(u)$ ($x$ has 4, the other 5 vertices in $B_u\setminus
		      N[m']$ have 3, $m'$ has at least 4, and each of the 3 vertices of
		      $B_u\cap N(m')$ has at least 2).

		      Suppose now there are 2 vertices $x_1,x_2$ of $B_u \setminus
		      N[m']$ with 4 neighbours in $N(u)$. These two additional edges
		      bring the total to 28. Thus, the total projection multiplicity is
		      7. There are at most 2 vertices projecting onto $x_1$ and 2
		      vertices projecting on $x_2$ (and none on the other vertices in
		      $B_u\setminus N[m']$). Thus, at least 3 vertices project onto
		      vertices in $N[m']$. This a contradiction, as $m'$ must have
		      exactly 2 neighbours in $N(u)$, otherwise there would be more than
		      28 edges between $B_u$ and $N(u)$. Considering that with
		      $\Delta=7$, no vertex of $B_u$ can have 5 or more neighbours in
		      $N(u)$, there are no cases left.
	\end{enumerate}

	In all possible cases, a contradiction was found. Thus, $c(G) \leq 3$.
\end{proof}


\section{Graphs with maximum degree 3}\label{maxdeg3section} In this section, we
consider the cop number of graphs with maximum degree 3. We start with the main
result of this section.

\begin{proposition}\label{propdegree3} If $G$ is a connected graph such that
	$\Delta \leq 3$ and $n \leq 20$, then $c(G)\leq 3$.
\end{proposition}

\begin{proof}
	We first prove the statement for $\delta \geq 2$. For $10 \leq n \leq 20$,
	we generate all graphs such that $\delta \geq 2$ and $\Delta\leq 3$, and
	then classify each graph according to its cop number. Both steps use the
	same software and scripts as in Section \ref{smallgraphssections}. We present
	the results in Table \ref{table:deg3}, which shows that no such graph with
	cop number at least 4 exists. We have also saved the precise 3-cop-win
	graphs.

	\begin{table}[h]
		\begin{tabular}{|c|c|c|c|c|c|}
			\cline{3-6}
			\multicolumn{2}{c}{} & \multicolumn{4}{|c|}{Cop number}
			\\
			\hline
			\rowcolor{gray!60}
			$n$                  & $G : \delta\geq 2, \Delta\leq 3$ & $1$  & $2$
			& $3$   & $\geq 4$ \\ \hline
			$10$                 & 458                              & 7    & 450
			& 1     & 0        \\ \hline
			$11$                 & 1353                             & 12   &
			1341      & 0     & 0        \\ \hline
			$12$                 & 4566                             & 21   &
			4543      & 2     & 0        \\ \hline
			$13$                 & 15530                            & 35   &
			15495     & 0     & 0        \\ \hline
			$14$                 & 56973                            & 63   &
			56901     & 9     & 0        \\ \hline
			$15$                 & 214763                           & 114  &
			214642    & 7     & 0        \\ \hline
			$16$                 & 848895                           & 211  &
			848622    & 62    & 0        \\ \hline
			$17$                 & 3454642                          & 388  &
			3454093   & 161   & 0        \\ \hline
			$18$                 & 14542574                         & 735  &
			14540858  & 981   & 0        \\ \hline
			$19$                 & 62871075                         & 1389 &
			62865352  & 4334  & 0        \\ \hline
			$20$                 & 279175376                        & 2664 &
			279147564 & 25148 & 0        \\ \hline
		\end{tabular}
		\caption{Cop number breakdown for connected subcubic graphs.}
		\label{table:deg3}
	\end{table}

	We now considers graphs which contain vertices of degree 1. We know that
	removing a vertex of degree 1 from a graph does not change the cop number
	nor the fact that it is connected (as the vertex of degree 1 is cornered by
	its neighbour, see Corollary \ref{removecorner}). We successively remove
	vertices of degree 1 from the graph. We eventually either get to a graph
	such that $\delta\geq 2$ and $n\geq 10$ (in which case the above results can
	now be applied) or we eventually get to a graph of order at most 9 (in which
	case we apply Theorem \ref{min3}).
\end{proof}

We will use this proposition to prove our main result in Section
\ref{mainsection}. The remainder of this section is not essential and mostly
concerns possible improvements and applications of this proposition.

Notwithstanding the slight improvement of considering $\delta\geq 2$, the
approach here is clearly far from optimal. The algorithm described in the
Section \ref{remainingcasessection} is an example of a possibly better strategy.
However, as we will see, this algorithm would not be the most efficient for
maximum degree 3 : to compute potential 4-cop-win graphs on 19 vertices, one
would still need to compute subcubic 3-cop-win graphs on 15 vertices. A
potentially more interesting algorithm for building possible 4-cop-win subcubic
graphs would consist in building graphs around long shortest paths (see
\cite[Lemma 4]{aigner_game_1984}, which describes how a cop can protect a
shortest path) by adding the desired number of other vertices and considering
all possible ways to add edges. Nonetheless, our exhaustive testing approach is
not without its advantages, as we can use it to gain further knowledge on the
cop number of small graphs.

In fact, Hosseini, Mohar and González Hermosillo de la Maza
\cite{hosseini_meyniels_2021} have recently shown that studying the cop number
of graphs with $\Delta\leq 3$ is of interest for the study of the cop number at
large. More precisely, they show that if it were true that the cop number of
subcubic graphs is in $O(\sqrt{n})$, then the cop number of general graphs is in
$O(n^\alpha)$ for all $\alpha>\frac{3}{4}$. In other words, proving Meyniel's
conjecture for subcubic graphs would substantially improve the best known upper
bound on the cop number. Hence, we consider that getting a distribution of the
cop-number of small subcubic graphs might be interesting, even if it is somewhat
skewed by adding the condition $\delta\geq 2$. Our computations show that not
only there are no 4-cop-win subcubic graphs on at most 20 vertices, but that
subcubic 3-cop-win graphs are overwhelmingly rare for these orders.

The exhaustive search approach also gives us progress on a related problem.
Arguably the best-known result on the game of cops and robbers is Aigner and
Fromme's proof that the cop number of any planar graph is at most 3, see
\cite{aigner_game_1984}. This yields the analogous question of finding the
minimum order of a 3-cop-win planar graph, and an enumeration of such graphs.
The smallest known planar 3-cop-win graph is the dodecahedral graph (see Figure
\ref{fig:dodecahedral_graph}), which has 20 vertices. It is easy to see that
this graph requires 3 cops, as it has girth 5 and is 3-regular. It has been
asked, first in \cite{andreae_pursuit_1986}, as well as in
\cite{bonato_topological_2017}, whether the dodecahedral graph is the unique
smallest 3-cop-win planar graph.

\begin{figure}[h]
	\begin{tikzpicture}[scale=1.4, dot/.style = {circle, fill, minimum size=#1,
					inner sep=0pt, outer sep=0pt}, dot/.default = 5pt
		]

		\node [dot] (1) at (1.548, 0.503) {}; \node [dot] (2) at (-1.134,
		-0.368) {}; \node [dot] (3) at (0., 1.628) {}; \node [dot] (4) at
		(0.957, -1.317) {}; \node [dot] (5) at (-1.548, 0.503) {}; \node [dot]
		(6) at (-0.957, -1.317) {}; \node [dot] (7) at (0., 2.466) {}; \node
		[dot] (8) at (1.449, -1.995) {}; \node [dot] (9) at (0.302, 0.416) {};
		\node [dot] (10) at (0.489, -0.159) {}; \node [dot] (11) at (-2.345,
		0.762) {}; \node [dot] (12) at (-1.45, -1.995) {}; \node [dot] (13) at
		(-0.489, -0.159) {}; \node [dot] (14) at (0.7, 0.965) {}; \node [dot]
		(15) at (1.133, -0.369) {}; \node [dot] (16) at (2.345, 0.762) {}; \node
		[dot] (17) at (-0.302, 0.416) {}; \node [dot] (18) at (0., -0.514) {};
		\node [dot] (19) at (-0.7, 0.965) {}; \node [dot] (20) at (0., -1.192)
		{};

		\draw (1) to (14); \draw (1) to (15); \draw (1) to (16); \draw (2) to
		(5); \draw (2) to (6); \draw (2) to (13); \draw (3) to (7); \draw (3) to
		(14); \draw (3) to (19); \draw (4) to (8); \draw (4) to (15); \draw (4)
		to (20); \draw (5) to (11); \draw (5) to (19); \draw (6) to (12); \draw
		(6) to (20); \draw (7) to (11); \draw (7) to (16); \draw (8) to (12);
		\draw (8) to (16); \draw (9) to (10); \draw (9) to (14); \draw (9) to
		(17); \draw (10) to (15); \draw (10) to (18); \draw (11)to (12); \draw
		(13) to (17); \draw (13) to (18); \draw (17) to (19); \draw (18) to
		(20);
	\end{tikzpicture}
	\caption[The dodecahedral graph]{The dodecahedral graph\footnotemark}
	\label{fig:dodecahedral_graph}
\end{figure}
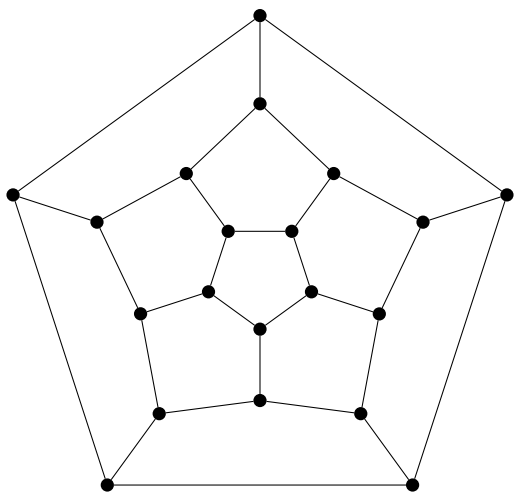

There are some partial results for this problem. In \cite{hosseini_game_2018},
Hosseini proves that a 3-cop-win planar graph of minimum order must be
2-connected. Furthermore, Pisantechakool and Tan have shown in
\cite{pisantechakool_conjecture_2017} that any planar graph on 19 or fewer
vertices must contain a winning position for 2 cops, although it has not been
proved that the cops can bring the game to this winning state. Using the
computations in the proof of Proposition \ref{propdegree3}, we are able to get
more evidence supporting the conjecture.

\begin{corollary}
	If $G$ is a connected planar graph such that $\Delta\leq 3$ and $n\leq 20$,
	then $c(G)\leq 2$, unless $G$ is the dodecahedral graph.
\end{corollary}
\begin{proof}
	We simply test the 3-cop-win graphs found in the proof of Proposition
	\ref{propdegree3} for planarity
	\cite{wolfram_research_inc_mathematica_nodate}. The only graph which was
	planar was the dodecahedral graph. The arguments relating to the fact that
	we only computed the graphs such that $\delta\geq 2$ still apply here.
\end{proof}

\footnotetext{Computer-generated drawing
	\cite{wolfram_research_inc_mathematica_nodate}.}


\section{Remaining cases}\label{remainingcasessection} In this section, we
consider the few remaining cases needed to prove that $M_4=19$, and also work
towards reducing the possible 4-cop-win graphs on 19 vertices. More precisely,
we consider graphs such that $n=17$ with $\Delta=4$, $n=18$ with $\Delta=4,5$
and $n=19$ with $\Delta=4$.

As in Section \ref{highdegreesection}, our main tool will be knowing that if a
graph $G$ is 4-cop-win, then for each vertex $u$, $c(G-N[u]) \geq 3$. We know
there are relatively few such graphs. In the cases of $\Delta=n-11$ or
$\Delta=n-12$, if $u$ is a vertex of maximum degree, we know that these $G-N[u]$
can only be the Petersen graph and cornered Petersen graphs (at least in the
connected case). As these are very few and very similar, we were able to build
structural properties that allowed us to show that $G$ had cop number at most 3.
Having somewhat a large maximum degree, a computational approach would have been
difficult due to the fact that there are too many possible edges to consider to
effectively construct all possible graphs.

For the cases we will now consider, a computational approach is possible as the
maximum degree is not too high. On the other hand, a proof similar to the one of
Section \ref{highdegreesection} for these cases would be difficult, although
certainly not impossible given a large amount of time. Indeed, there are too
many possible choices for $G-N[u]$, and they are not similar enough. As most
graphs found in Lemma \ref{computationn1213} contain the Petersen graph as an
induced subgraph, modifying the strategy to take these vertices into account
could seem reasonable. However, we saw that adding even a single corner to the Petersen
graphs yields significant complications for the proof. Adding multiple vertices
to the Petersen graph would likely be much more complex. Furthermore, some of
the graphs found in Lemma \ref{computationn1213} do not contain the Petersen
graph as an induced subgraph at all, and would need to be considered separately.
For these reasons, we have mostly investigated the computational approach.

Our goal is to build graphs which are possibly 4-cop-win: graphs $G$ for which
we cannot say that $c(G)\leq 3$ simply by looking at $G-N[u]$ for the vertices
$u$ of maximum degree, which we will call \emph{candidate 4-cop-win graphs}.

The simplest idea, which we have briefly discussed in Section
\ref{smallgraphssections}, would be simply to consider a 3-cop-win graph $G'$ on
12 or 13 vertices, add a vertex $u$ of chosen maximum degree and its
neighbourhood, and then look at every possible ways of joining $N(u)$ to $G'$ by
respecting the maximum degree condition. Then we could compute the cop number of
each of these graphs. Even by reducing the number of cases by isomorphism, the
number of graphs to consider is massive, especially in the case $\Delta=5$, so
we must be a tad smarter. We present the Merging Algorithm as a way to generate
candidate 4-cop-win graphs, which we then test using a standard cop-number
algorithm.

We briefly introduce some notation. In general, when considering a graph $G$ and
a vertex $u$, the degree of $u$ will always refer to the degree of $G$ in $u$.
If we want to discuss the degree of $u$ in some induced subgraph $H$, we will
refer to it as the $H$-degree of $u$. In general, if we say "there exists a
vertex of $H$-degree $r$", then we are also implicitly stating that this vertex
is in $H$. Given some function $\phi$ defined on some set containing $T$, we
will use the notation $\phi(T)=\{\phi(t):t\in T\}$.

\subsection{Presentation of the Merging Algorithm}

\subsubsection{Quick Overview}

Our approach to build candidate 4-cop-win graphs $G$ will be the following. Let
$v_1$ and $v_2$ be non-adjacent vertices, which we will in general choose to be
a pair with the highest possible degrees. Denoting $G_1=G-N[v_2]$ and
$G_2=G-N[v_1]$, we know that if $G$ is to be a candidate 4-cop-win graph, the
cop number of $G_1$ and $G_2$ must have cop number at least 3 (we will see later
that we can consider $G_1$ and $G_2$ to be connected).

In Section \ref{smallgraphssections}, we computed the 3-cop-win graphs on
specific numbers of vertices and with some maximum degree conditions. As we have
a few different cases to consider, the possible choices of $G_1$ and $G_2$ will
vary. For now, we can simply denote by $\mathcal L_1$ and $\mathcal L_2$ some
sets of 3-cop-win graphs. Our goal is to determine every possible graph $G$,
with maximum degree $\Delta$, for which $G_1\in \mathcal L_1$ and $G_2\in
\mathcal L_2$. We will call the process the Merging Algorithm, which we will now
detail.

\subsubsection{Input of the Algorithm}\label{algoinput}

Integers $n,D_1,D_2$ and sets of (isomorphism classes of) graphs $\mathcal L_1$
and $\mathcal L_2$, such that
\begin{enumerate}
	\item the graphs in $\mathcal L_1$ and $\mathcal L_2$ are connected,
	      3-cop-win and have maximum degree at most $\Delta$, and
	\item the graphs in $\mathcal L_1$ have $n-D_2-1$ vertices and the graphs in
	      $\mathcal L_2$ have $n-D_1-1$ vertices.
\end{enumerate}

\subsubsection{Output of the Algorithm}\label{algooutput}

The algorithm returns all connected graphs $G$ on $n$ vertices and maximum
degree $\Delta=D_2$ which contain a pair of non-adjacent (and distinct) vertices
$v_1$ and $v_2$, with the following 4 properties. Denote $G_1=G-N[v_2]$ and
$G_2=G-N[v_1]$. Then,
\begin{enumerate}
	\item $v_1$ and $v_2$ have degree respectively $D_1$ and $D_2$,
	\item $G_1\in \mathcal L_1$ and $G_2\in \mathcal L_2$,
	\item for all other vertices $u$ of maximum degree ($d(u)=\Delta$),
	$G-N[u]\in \mathcal L_1$, and
	\item if $D_1<D_2$, then the set of vertices of $G$ of maximum degree forms
	      a clique and $v_1$ and $v_2$ have at least 1 common neighbour.
\end{enumerate}

Isomorphic graphs may be omitted from the results, as we are not interested in
the precise labellings of the graphs.

\subsubsection{Required definitions}
We define a few concepts which will be useful in the description of the
algorithm.

A \emph{partially-constructed graph} is a triple $(\widehat G,\mathcal F,R)$
where $\widehat G$ is a graph, $\mathcal F$ is a partition of $V(\widehat G)$,
and $R\subseteq V(\widehat G)$. The graph $\widehat G$ is a graph to which we
might still add edges in later steps of the algorithm. The set $\mathcal F$ will
be used to contain some information on how $\widehat G$ was constructed. We will
use $R$ to denote some set of vertices which must have degree strictly smaller
than $D_2=\Delta$, which will be useful to reduce the number of cases we need to
verify.

We say $\phi$ is a \emph{strong isomorphism} between $(\widehat G,\mathcal F,R)$
and $(\widehat G',\mathcal F,R)$ (we suppose both graphs have the same extra
structure given to them, in particular $V(\widehat G)=V(\widehat G')$) if $\phi$
is an isomorphism between $\widehat G$ and $\widehat G'$ such that $\phi(S)=S$
for each $S\in \mathcal F$ and such that $\phi(R)=R$, i.e. $\phi$ is an
isomorphism that preserves the extra structure we attach to our graphs.

Let $G_1\in \mathcal L_1$. We define an equivalence relation on $V(G_1)$ as
follows: $u\sim_{G_1}v$ if there is some automorphism $\psi$ of $G_1$ such that
$\psi(u)=v$. Then, we can define a total ordering $\leq_{G_1}$ on
$V(G_1)/\sim_{G_1}$ by comparing the vertices themselves as
follows: if $u,v$ are two distinct (equivalence classes of) vertices, we set 
$u<_{G_1} v$ if $d_{G_1}(u)< d_{G_1}(v)$, and if $u,v$ have the same degree, 
then we choose the order arbitrarily.

\subsubsection{Phase 1 of the Algorithm}\label{subsec:phase_1_merging}

First choose some $G_1$ and $G_2$ from $\mathcal L_1$ and $\mathcal L_2$
respectively. The rest of the algorithm will be repeated for each possible
choice of $G_1$ and $G_2$. Also choose strictly positive integers $d_1$ and
$d_2$ such that $D_2-d_2=D_1-d_1$, $d_1\leq D_1$ and $d_2\leq D_2$. Again, every
possible choice will be considered. When $D_1<D_2$, it will suffice to
pick $d_1$ such that $d_1<D_1$ (and thus $d_2<D_2$).

Then consider every possible choice of $v_1\in V(G_1)$ and $v_2\in V(G_2)$ such
that $v_1$ has $G_1$-degree $d_1$ and $v_2$ has $G_2$-degree $d_2$ (of course,
$v_1$ and $v_2$ can be considered up to automorphism in $G_1$ and in $G_2$). For
each choice of vertices, consider every possible way of identifying $G_1-N[v_1]$
and $G_2-N[v_2]$ by computing every isomorphism between these graphs. If there
are none, this branch of the algorithm simply doesn't yield a graph. For each
identification, the graphs may be merged by union, keeping the closed
neighbourhoods of $v_1$ and $v_2$ distinct.

In the case $D_1=D_2$, let $R=\{u\in V(G_1) : u>_{G_1}v_1\}$ be the set of
vertices of $G_1$ strictly greater than $v_1$ in the ordering defined above. In
the second main case of the algorithm, when $D_1<D_2$, set $R=V(G_1)$. Note that as $R$ 
is a subset of the vertices of $V(G_1)$, it is also a subset of the vertices in the graph 
resulting from the identification above. If the
merging process above has created vertices of degree greater than $D_2=\Delta$,
the graph is thrown out: as the rest of the algorithm can only raise the degree
again, it would yield graphs we do not want to consider. Similarly, throw out the
graphs in which any vertex in $R$ has maximum degree. Finally, add
$D_2-d_2=D_1-d_1$ common neighbours to $v_1$ and $v_2$ (these are new vertices).
Now $v_1$ has degree $D_1$ and $v_2$ has degree $D_2$. Note that in the case
$D_1<D_2$, we limited $d_1$ to $d_1<D_1$, so $v_1$ and $v_2$ have at least 1
common neighbour.

Denote by $\widehat G$ some graph resulting from these operations. It is easily
seen in Figure \ref{fig:mergingpart1} that the construction implicitly
partitions the vertices into 6 sets (or "types" of vertices) : $$\mathcal
F=\{\{v_1\},N(v_1)\setminus N(v_2),(N[v_1]\cup N[v_2])^c,N(v_1)\cap
N(v_2),N(v_2)\setminus N(v_1),\{v_2\}\}$$ where the complement is taken relative to
$V(\widehat G)$. We call the partially-constructed graph $(\widehat
G,\mathcal F,R)$ a \emph{base graph}. We will sometimes call $\widehat G$ itself
a base graph, as the remaining structure is usually clear.

Suppose $\widehat G$ and $\widehat G'$ are generated in Phase 1 of the algorithm
with the same properties, that is, they are produced from same algorithm inputs and
the same choices of $G_1,G_2,d_1,d_2,v_1,v_2$ but by choosing a different identification 
in the merging step. In this case, it is clear that $\mathcal F,R$ are identical for both.
If $(\widehat G,\mathcal F,R)$ and $(\widehat G',\mathcal F,R)$ are strongly
isomorphic, we can consider these graphs to be duplicates: only one needs to be
sent to Phase 2 of the algorithm. We will see later that, up to isomorphism,
these strongly isomorphic base graphs necessarily result in isomorphic candidate
4-cop-win graphs.

The result of Phase 1 should then be some (usually large) collection of base
graphs, as we have to consider every choice of $G_1,G_2,d_1,d_2,v_1,v_2$ and
identification of $G_1-N[v_1]$ and $G_2-N[v_2]$.

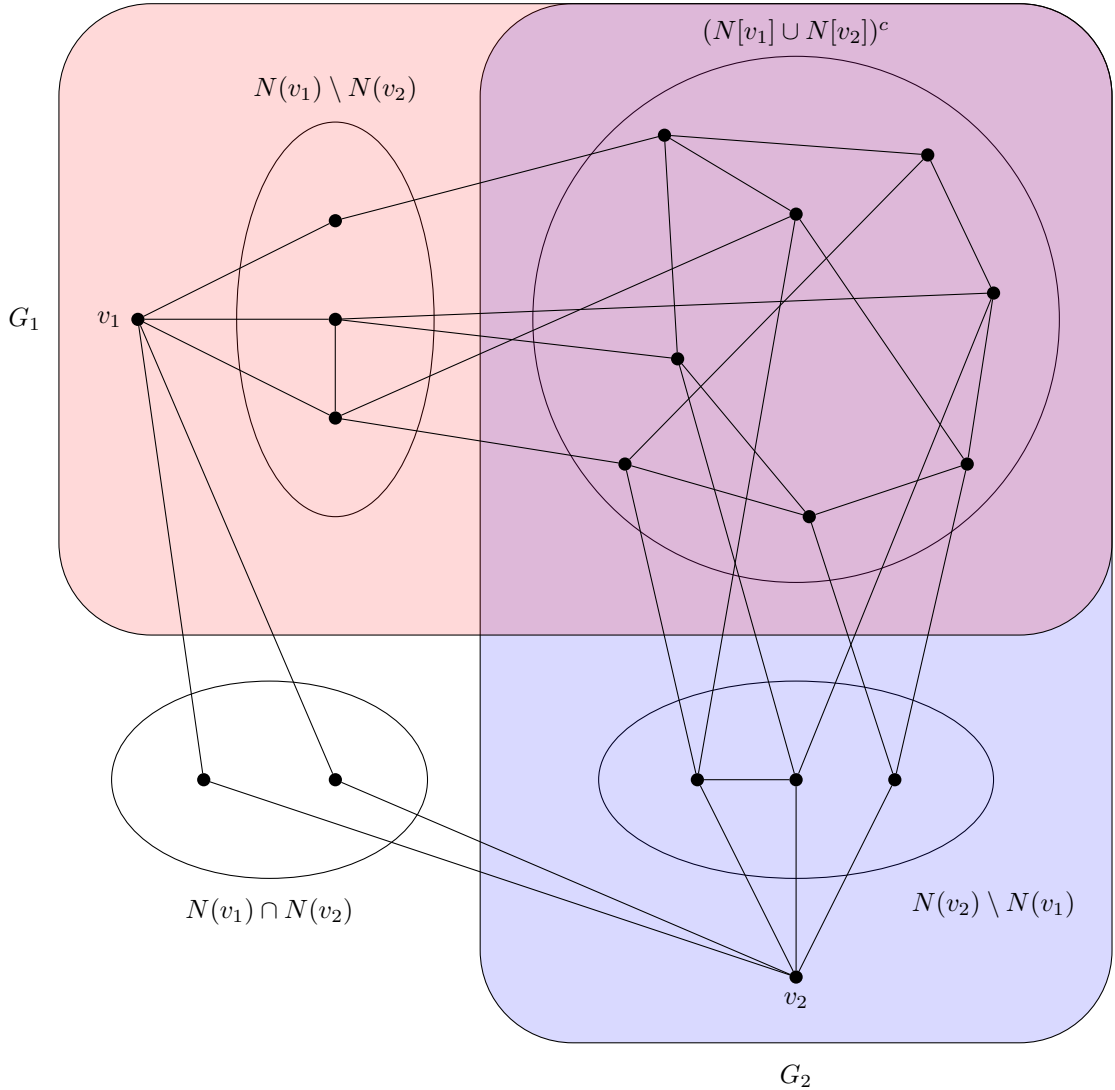
\begin{figure}[h]
	\begin{tikzpicture}[scale=1.75, dot/.style = {circle, fill, minimum size=#1,
					inner sep=0pt, outer sep=0pt}, dot/.default = 5pt
		]


		\draw (-4,1.5) ellipse (1.2 and 0.75);

		\draw (0,1.5) ellipse (1.5 and 0.75);

		\draw (-3.5,5) ellipse (0.75 and 1.5);

		\draw (0,5) ellipse (2 and 2);

		\draw[rounded corners=35pt, fill=blue,fill opacity=0.15] (-2.4, -0.5)
		rectangle (2.4, 7.4) {}; \draw[rounded corners=35pt, fill=red,fill
		opacity=0.15] (-5.6, 2.6) rectangle (2.4, 7.4) {};

		\node [dot, label=below:$v_2$] (1) at (0, 0) {}; \node [dot,
		label=left:$v_1$] (2) at (-5, 5) {};

		\node [dot] (3) at (-3.5, 1.5) {}; \node [dot] (4) at (-4.5, 1.5) {};

		\node [dot] (5) at (-0.75, 1.5) {}; \node [dot] (6) at (0, 1.5) {};
		\node [dot] (7) at (0.75, 1.5) {};

		\node [dot] (8) at (-3.5, 5) {}; \node [dot] (9) at (-3.5, 5.75) {};
		\node [dot] (10) at (-3.5, 4.25) {};

		\node [dot] (11) at (-1, 6.4) {}; \node [dot] (12) at (1, 6.25) {};
		\node [dot] (13) at (0, 5.8) {}; \node [dot] (14) at (1.5, 5.2) {};
		\node [dot] (15) at (1.3, 3.9) {}; \node [dot] (16) at (0.1, 3.5) {};
		\node [dot] (17) at (-1.3, 3.9) {}; \node [dot] (18) at (-0.9, 4.7) {};

		\draw (1) to (3); \draw (1) to (4);

		\draw (2) to (3); \draw (2) to (4);

		\draw (1) to (5); \draw (1) to (6); \draw (1) to (7);

		\draw (2) to (8); \draw (2) to (9); \draw (2) to (10);

		\draw (5) to (6); \draw (8) to (10);

		\draw (5) to (17); \draw (5) to (13); \draw (6) to (18); \draw (6) to
		(14); \draw (7) to (16); \draw (7) to (15); \draw (9) to (11); \draw (8)
		to (14); \draw (8) to (18); \draw (10) to (17); \draw (10) to (13);

		\draw (11) to (12); \draw (11) to (13); \draw (11) to (18);

		\draw (12) to (14); \draw (12) to (17);

		\draw (14) to (15);

		\draw (15) to (13); \draw (15) to (16);

		\draw (16) to (18); \draw (16) to (17);

		\node at (0,-0.75) {$G_2$}; \node at (-5.86,5) {$G_1$};

		\node at (-3.5,6.75) {$N(v_1)\setminus N(v_2)$};

		\node at (1.5,0.55) {$N(v_2)\setminus N(v_1)$};

		\node at (-4,0.5) {$N(v_1)\cap N(v_2)$};

		\node at (0,7.18) {$(N[v_1]\cup N[v_2])^c$};
	\end{tikzpicture}
	\caption{Example from Phase 1 of the Merging Algorithm. Here, the base graph was generated using parameters $n=18$, $D_1=D_2=\Delta=5$ and $d_1=d_2=3$.}
	\label{fig:mergingpart1}
\end{figure}

\subsubsection{Phase 2 of the Merging Algorithm}\label{subsec:phase_2_merging}

The goal of this phase is to complete the construction of the 4-cop-win
candidates graphs by adding edges to the base graphs. We will run Phase 2 for
every base graph from Phase 1. Suppose we have some base graph $(\widehat
G,\mathcal F,R)$. We know that we do not want to add any edge such that both
ends are in $G_1$ or both ends in $G_2$, as we want these to be induced subgraphs
of $G$. Furthermore, we have already created all incident edges to either $v_1$
or $v_2$, by giving them the desired number of neighbours. Thus, we only need to
consider adding edges which are either

\begin{enumerate}
	\item between $N(v_1)\cap N(v_2)$ and $\{v_1,v_2\}^c$, including edges with
	      both ends in $N(v_1)\cap N(v_2)$, or
	\item between $N(v_2)\setminus N(v_1)$ and $N(v_1)\setminus N(v_2)$.
\end{enumerate}

We fix an order on the vertices of $N(v_2)$ as follows : first the vertices in
$N(v_1)\cap N(v_2)$, then those in $N(v_2)\setminus N(v_1)$. Denote $a_i$ to be
$i$-th vertex in this order, which will be considered at the $i$-th step of
Phase 2.

At the $i$-th step of Phase 2, a new partially-constructed graphs is created for
each way of adding edges incident to $a_i$ of the form described above, with the 
following restrictions:
\begin{enumerate}
	\item not to add an edge if one of its end vertices already has degree
	$\Delta$,
	\item not to add an edge if one of its end vertices is in $R$ and adding
	this edge would bring that vertex to degree $\Delta$, and
	\item not to add an edge between $a_ia_j$ if $i>j$.
\end{enumerate}
Then, for each of these newly created graphs, repeat the process by moving on to
$a_{i+1}$, and so on. Hence restriction (3) comes from the fact that such an
edge $a_ia_j$ has already been considered at the $j$-th step.

In this phase, we will also reduce some cases by strong isomorphism. Let
	$A_i=\{a_1,\dots,a_i\}$. If $a_i\in N(v_1)\cap N(v_2)$, let $$\mathcal
	T_i=\{N(v_1)\cap N(v_2)\cap A_i,(N(v_1)\cap N(v_2))\setminus
	A_i,N(v_2)\setminus N(v_1)\}.$$ Otherwise, if $a_i\in N(v_2)\setminus
	N(v_1)$, then let $$\mathcal T_i=\{N(v_1)\cap N(v_2),(N(v_2)\setminus
	N(v_1))\cap A_i,(N(v_2)\setminus N(v_1))\setminus A_i\}.$$Finally, define
	$$\mathcal F_i=\{\{v_1\},N(v_1)\setminus N(v_2),(N[v_1]\cup
	N[v_2])^c,\{v_2\}\}\cup \mathcal T_i.$$ In other words, we refine $\mathcal
	F$ by distinguishing the vertices for which we have already considered
	adding extra edges.

Hence, we will attach to the graphs produced in $i$-th step of Phase 2 this
extra structure, meaning the $i$-th step produces partially-constructed graphs
of the form $(\widehat G', \mathcal F_i, R)$. Note that $\mathcal F_i$ only
depends on the base graph and at the step number in the edge adding procedure,
and not the choices of edges added in steps $1,\dots,i$. Hence, we may compare
partially-constructed graphs and remove strongly isomorphic repetitions (we
remind the reader that this is only for those constructed from the same base graph).
As this procedure is often lengthy, we only remove graphs by
strong-isomorphism in some cases.

In the case $D_1<D_2$, if at any point the (partially-constructed) graph
contains non-adjacent vertices both of degree $\Delta$, we can throw out this
graph, since these vertices will necessarily remain non-adjacent in final
graphs.

After considering every possible way of adding edges (after the $|N(v_2)|$-th
step), all graphs $G$ such that there is some vertex $u$ of maximum degree where
$G-N[u]$ is not in $\mathcal L_1$ are thrown out (by construction, we do not need
to verify this for $v_1$ and $v_2$). From the remaining graphs, we also remove
isomorphic graphs. Note that since we have split up the computations in many
pieces, we only compare for isomorphism graphs which were generated from the
same base graph.

\subsubsection{Comments on the Merging Algorithm}
We note that in practice, in order to verify strong isomorphism, we do not need to 
test whether $\phi(R)=R$ holds. Indeed, as $\phi$ preserves in particular
$\{v_1\}$, $N(v_1)\setminus N(v_2)$ and $V(G)\setminus(N[v_1]\cup N[v_2])$, we
know that $\phi$ restricted to the vertices of $G_1$ is an automorphism of
$G_1$. Recall that $R$ was constructed using the ordering $\leq_{G_1}$ on
$V(G_1)/\sim_{G_1}$ (when $D_1=D_2$) or was chosen to be $V(G_1)$ (when
$D_1<D_2$). Hence $R$ is always closed under under automorphism of $G_1$, so in
particular $\phi(R)=R$.

In Phase 1, one particular consequences of having a set $R$ is the following.
Suppose $D_1=D_2$ and $G_1$ already (before merging) contains multiple vertices
of degree $D_1=D_2=\Delta$. Let $u$ be some vertex maximal relative to
$\leq_{G_1}$ (it is unique up to automorphisms). Then, in Phase 1 when $v_1$ is 
not $u$, no graph will be generated. Indeed, in this case $u \in R$
but $u$ necessarily has degree $\Delta$ after merging, a case we have excluded.
For this reason, when $G_1$ contains a vertex of degree $\Delta$, we do not need
to try multiple choices of $v_1$. In particular, when $\Delta(G_1)=\Delta$, we
can only consider the case $d_1=D_1=\Delta$.

\subsubsection{Validity of the Merging Algorithm}

Considering that the algorithm itself is relatively straightforward, we do not
present a complete proof of the validity of the algorithm. We however present a
few key points.

Consider a graph $G$ respecting the conditions described in the Section
\ref{algooutput}. Our goal is to show that $G$, or a graph isomorphic to $G$, is
constructed by the Merging Algorithm. Choose $v_2$ to be any vertex of degree
$D_2$ in $G$ such that $V(G)\setminus N[v_2]$ contains at least one vertex
respecting the conditions for $v_1$ in Section \ref{algooutput}. Denote by $S$
this non-empty set of possible choices for $v_1$. In the case $D_1<D_2$, we
choose $v_1$ to be any vertex of $S$. If $D_1=D_2$, then choose $v_1$ to be
maximal in $S$ relative to the ordering $\leq_{G_1}$. Then, write
$G_1=G-N[v_2]$, $G_2=G-N[v_1]$, $D_1=d_G(v_1)$, $D_2=d_G(v_2)$,
$d_1=d_{G_1}(v_1)$ and $d_2=d_{G_2}(v_2)$. It is easy to verify that
\begin{align*}
	d_2-d_1 & =d_{G_2}(v_2)-d_{G_1}(v_1)                                                            \\
	        & =|N_{G-N[v_1]}(v_2)|-|N_{G-N[v_2]}(v_1)|                                              \\
	        & =|N(v_2)\setminus (N[v_1]\cap N(v_2))|-|N(v_1)\setminus (N(v_1)\cap N[v_2])|          \\
	        & =|N(v_2)\setminus (N(v_1)\cap N(v_2))|-|N(v_1)\setminus (N(v_1)\cap N(v_2))|          \\
	        & =\left(|N(v_2)|-|N(v_1)\cap N(v_2)|\right)-\left(|N(v_1)|- |N(v_1)\cap N(v_2)|\right) \\
	        & =|N(v_2)|-|N(v_1)|                                                                    \\
	        & =D_2-D_1.
\end{align*}

Thus, the pair of degrees $d_1$ and $d_2$ is indeed considered in the Merging
Algorithm. It is then easy to see that the Phase 1 of the algorithm will have
some branch considering this exact choice of $G_1,G_2,d_1,d_2,v_1,v_2$ and the
actual isomorphism between $G_1-N[v_1]$ and $G_2-N[v_2]$. Note that in the case
$D_1=D_2$, choosing $v_1$ as maximal in $S$ relative to $\leq_{G_1}$ guarantees
that all vertices of $G-N[v_2]$ which are strictly greater to $v_1$ in this
order do not have maximal degree in $G$, which is consistent with how we use $R$
in the algorithm.

Then, as in Phase 2 of the algorithm we consider adding every possible edge not
totally contained in $G_1$ or $G_2$ (precisely the edges between $N(v_1)\cap
N(v_2)$ and $\{v_1,v_2\}^c$, or between $N(v_2)\setminus N(v_1)$ and
$N(v_1)\setminus N(v_2)$) while still respecting some degree conditions, $G$
must have been constructed.

Other steps in the algorithm, such as requiring that $G-N[u]\in \mathcal L_1$
for $u$ of maximum degree, or in the case $D_1<D_2$ removing graphs with
non-adjacent vertices of maximum degree, are consistent with Section
\ref{algooutput}.

In fact, what we have stated so far is only true up to the removal of
partially-constructed graphs using strong isomorphism. Indeed, although we claim
a graph isomorphic to $G$ has necessarily been constructed, the graph $G$ might
not have been constructed itself. To see this, suppose that $(\widehat
G,\mathcal F,R)$ was thrown out at some step because it was strongly isomorphic
to $(\widehat G',\mathcal F,R)$ (where the strong isomorphism is $\phi:
V(\widehat G)\rightarrow V(\widehat G')$). We wish to prove that every graph
that would be constructed from $(\widehat G,\mathcal F,R)$ (had it not been
pruned) can be constructed from $(\widehat G',\mathcal F,R)$, up to
isomorphism.

Suppose that the edges we want to add are $E_1=\{w_1z_1,w_2z_2,\dots,w_kz_k\}$,
which we can note as $G=\widehat G+E_1$. Then, it is simple to see that $\phi$
is also an isomorphism between $G=\widehat G+E_1$ and $G'=\widehat G'+E_2$,
where $E_2=\{\phi(w_1)\phi(z_1),\phi(w_2)\phi(z_2),\dots,\phi(w_k)\phi(z_k)\}$.
It remains to see that Phase 2 algorithm allows us to add the edges in $E_2$ to
$\widehat G'$, that is that $\widehat G'+E_2$ will actually be constructed. For
now, suppose that none of the partially-constructed graph in the intermediate
steps between $(\widehat G',\mathcal F,R)$ and $G'$ are throw out because of
some other strong isomorphism.

As $\phi$ is a strong isomorphism, the sets $\{v_1\},N(v_1)\setminus
	N(v_2),(N[v_1]\cup N[v_2])^c,N(v_1)\cap N(v_2),N(v_2)\setminus
	N(v_1),\{v_2\}$ are preserved by $\phi$ (the permutations between vertices
	are inside these sets). Hence, if $w_iz_i$ is between $N(v_1)\cap N(v_2)$
	and $\{v_1,v_2\}^c$, then $\phi(w_i)\phi(z_i)$ is between $N(v_1)\cap
	N(v_2)$ and $\{v_1,v_2\}^c$ as well. The same goes for the other type of
	edge we add.  In Phase 2, we also required that $\phi$ preserves the
	vertices in $N(v_1)\cap N(v_2)$ and in $N(v_2)\setminus N(v_1)$ to which we
	have already added (possibly zero) edges (recall the refinement $\mathcal
	F_i$ or $\mathcal F$). Hence, if some vertex $a\in N(v_2)$ has not yet been
	considered in Phase 2 by the time we throw out $(\widehat G,\mathcal F,R)$,
	neither has $\phi(a)$ in $(\widehat G,\mathcal F,R)$. This excludes the
	possibility that because of the strong isomorphism some vertex which had to
	eventually be considered never is. When adding edges, some degree conditions
	also had to be respected. In particular, there is a fixed maximum degree,
	and the vertices in $R$ must have degree strictly smaller than the maximum
	degree. Let $x$ be a vertex in $V(\widehat G)$. As $\phi$ is an isomorphism,
	the degrees of $x$ and $\phi(x)$ are identical. As $\phi$ is a strong
	isomorphism, $x\in R$ if and only if $\phi(x)\in R=\phi(R)$. Hence, the
	maximum number of edges that we can add to a vertex $x$ is the same as
	$\phi(x)$. Therefore, as the edge $w_iz_i$ was authorized to be added to $\widehat
	G$, there will be no degree restrictions forbidding the edge
	$\phi(w_i)\phi(z_i)$ to be added to $\widehat G'$.

This shows that for every $G$ that would be constructed from $(\widehat
G,\mathcal F,R)$, there is some $G'$ constructed from $(\widehat G',\mathcal
F,R)$ which is isomorphic to $G$, proving that throwing away $(\widehat
G,\mathcal F,R)$ is valid. However, this supposed that during the
steps between $(\widehat G',\mathcal F,R)$ and $G'$, none of the intermediate
partially-constructed graphs were thrown out by some other strong
isomorphism. We may now suppose the contrary. Consider that at some point,
the intermediary graph $(\widehat G'+E_2',\mathcal F,R)$ (in particular $E_2'\subset
E_2$) was thrown out in favour of some other graph $(\widehat G'',\mathcal F,R)$. While
neither $G$ nor $G'$ was constructed, we wish to show there is some $G''=G'+E_3$
constructed from $(\widehat G'',\mathcal F,R)$ which is isomorphic to $G'$ (and
to $G$). We can apply the same argument as above to show that one such $G''$
exists, assuming that between $\widehat G''$ and $G''$ no intermediate graph is
removed by strong isomorphism. If it is, then by repeating this argument (at
most $|N(v_2)|$ times, the number of iterations), we
will eventually find some graph constructed that was isomorphic to $G$.

We note that even if $(\widehat G,\mathcal F,R)$ and $(\widehat G',\mathcal
F,R)$ are strongly isomorphic graphs, the partially-constructed graphs they lead to one
step later are not necessarily strongly isomorphic, in particular as the
ordering of the vertices $a_i$ is not preserved by $\phi$ (not completely, although the
information on whether $a_i$ has been considered or not by Phase 2 yet is
preserved). Hence, the proof above requires finding a correspondance between the
final graphs they produce, and not tracking the constructions step by step.

\subsection{Results}

We will now use the Merging Algorithm to build all possible 4-cop-win graphs.
Our implementation of the algorithm is done in the Wolfram language
\cite{wolfram_research_inc_mathematica_nodate}.

To lighten the proof, we will use the following notation. Define a property $P$
with the usual definition: a property $P$ is a function from a set to a Boolean
value. For instance, if $C_3$ is the graph property of being 3-cop-win, then
$C_3(\peter_0)$ is whether the Petersen graph is 3-cop-win (which is true).

\begin{proposition}\label{propremainingcases} Let $G$ be a connected graph such
	that either
	\begin{enumerate}
		\item $n=17$ and $\Delta=4$,
		\item $n=18$ and $\Delta\in \{4,5\}$, or
		\item $n=19$ and $\Delta=4$.
	\end{enumerate}
	If every proper induced connected subgraph $H$ of $G$ respects $c(H)\leq 3$,
	then $c(G)\leq 3$, unless $G$ is the Robertson graph (see Figure
	\ref{fig:petersenrobertson_graph}).
\end{proposition}
\begin{proof}
	Let $u$ be any vertex of $G$. We know that $G-N[u]$ has at most $n-2 \leq
	17$ vertices. If $G-N[u]$ is disconnected, it must contain at least one
	component $K$ which has at most 8 vertices. Then, Theorem \ref{min3} implies
	that $c(K)\leq 2$. Furthermore, our hypothesis implies that $c(G-K)\leq 3$,
	as $G-K$ is necessarily a connected induced subgraph of $G$. By Corollary
	\ref{retractcomponents}, we get that $c(G)\leq \max \{c(G-K),c(K)+1\} \leq
	3$. Thus, for the remainder of the proof, we assume that for every vertex
	$u$, $G-N[u]$ is connected. Likewise, we can assume that $G$ does not
	contain a corner $x$. Indeed, $G-x$ is necessarily connected and has cop
	number at most 3, therefore Corollary \ref{removecorner} implies that $G$
	also has cop number at most 3. We may also assume that $c(G-N[u])=3$: if
	$c(G-N[u]) \leq 2$, placing a stationary cop on $u$ implies that $c(G)\leq
	3$.

	We can bring together these assumptions by defining the property $M(G)$ as
	follows : $G$ is a graph respecting the hypotheses of the proposition and
	such that $G-N[u]$ is a connected 3-cop-win graph for every vertex $u$ of
	$G$ and such that $G$ does not contain a corner. By the previous discussion,
	it suffices to show the proposition for graphs respecting $M$.

	We now define property $P_1$. A graph $G$ is said to have property $P_1$ if
	$G$ contains two non-adjacent vertices of maximum degree $\Delta$. We use
	the Merging Algorithm to generate all graphs $G$ such that $M(G)$ holds and
	that respect property $P_1$, and then compute their cop numbers. More
	precisely, we choose $n$ and $\Delta$ according to the case we are
	considering, we set $D_1=D_2=\Delta$, and choose $\mathcal L_1=\mathcal L_2$
	to be the set of 3-cop-win graphs on $n-\Delta-1$ vertices with maximum
	degree at most $\Delta$, as computed in Lemma \ref{computationn1213}. We
	note that the Merging Algorithm computes a somewhat larger class of graphs
	than we want. In particular, the Merging Algorithm does not exclude graphs
	which contain corners and in its last step only tests whether $G-N[u]$ is
	3-cop-win (if $G-N[u]\in \mathcal L_1$) for vertices of maximum degree.

	\begin{table}[h]
		\centering
		\begin{tabular}{|c|c|c|c|c|c|c|c|c|c|c|}
			\cline{8-11}
			\multicolumn{7}{c}{} & \multicolumn{4}{|c|}{Cop number}
			\\
			\hline
			\rowcolor{gray!60}
			$\Delta$             & $n$                              & $\Delta_1$
			                     & $G_1$ & $d_1$                            &
			                     Base graphs & Final graphs & $1$   & $2$ & $3$
			                     & $\geq 4$                                  \\
			                     \hline
			\multirow{5}{*}{4}   & \multirow{2}{*}{17}              & $4$
			                     & $78$ & $4$                              & 123
			                     & 0            & 0     & 0 & 0
			                     & 0                                         \\
			                     \cline{3-11} &                        & $3$ &
			                     $2$                              & $3$
			                     & 10           & 0     & 0 & 0
			                     & 0           & 0                           \\
			                     \cline{2-11} & 18 & $4$
			                     & $1105$      & $4$          & 1668  & 0 & 0
			                     & 0           & 0            & 0            \\
			\cline{2-11}         & \multirow{2}{*}{19}              & $4$
			                     & $16514$ & $4$                              &
			                     33785       & 3            & 0     & 0 & 0
			                     & 3                                         \\
			                     \cline{3-11} &                        & $3$ &
			                     $9$                              & $3$
			                     & 911          & 0     & 0 & 0
			                     & 0           & 0                           \\
			                     \hline
			\multirow{6}{*}{5}   & \multirow{6}{*}{18}              & $5$
			                     & $93$ & $5$                              &
			                     14232       & 24416        & 0     & 5484 &
			                     18932                            & 0
			                     \\ \cline{3-11} &                        &
			                     \multirow{4}{*}{4}   & \multirow{4}{*}{78}
			                     & $4$         & 10062        & 39318
			                     & 0                                & 7410
			                     & 31908        & 0 \\ \cline{5-11} &
			                     &                       & & $3$
			                     & 534         & 18645        & 0     & 3455 &
			                     15190                            & 0
			                     \\ \cline{5-11} &                        & &
			                     & $2$         & 111          & 24238 & 0 & 1494
			                     & 22744       & 0                           \\
			                     \cline{5-11} & &
			                     &             & $1$          & 88    & 698809
			                     & 0                                & 82882
			                     & 615927       & 0 \\ \cline{3-11} &
			                     & $3$                   &  $2$ & $3$
			                     & 22          & 12778        & 0     & 4960 &
			                     7818                             & 0
			                     \\ \hline
		\end{tabular}
		\caption{Results of the first wave of computations using the Merging Algorithm. It presents the counts for the graphs built with the property that they contain 2 non-adjacent vertices of maximum degree. In particular, $d_1=d_2$ and $\Delta=D_1=D_2$. Furthermore, $G_1$ is chosen with maximum degree $\Delta_1$.}
		\label{tab:algoresultspart1}
	\end{table}

	The summary results are presented in Table \ref{tab:algoresultspart1}. For
	more detail, we also split up the graphs relative to the various possible
	maximum degrees of $G_1$, although we of course always merge with all of the
	possible graphs $G_2$, not only the $G_2$ with the same maximum degree. We
	note that there are no 3-cop-win graphs with maximum degree 3 on 13 vertices
	(which can also be seen in Table \ref{table:deg3}), and that the 3-cop-win
	graphs with maximum degree 3 on 12 and 14 vertices are 3-regular (and thus
	the only possible value of $d_1$ is 3).

	We note that the 4-cop-win graphs found on 19 vertices are actually all
	copies of the Robertson graph. In fact, the 3 copies correspond to 3
	different choices of $G_1$ which can yield the Robertson graph.

	With these results, we will then only consider graphs which do not have
	property $P_1$. In other words, the graphs left to consider are those such
	that the set of vertices of maximum degree of $G$ forms a clique. This is a
	very restrictive property, and will be very useful.

	We see that graphs $G$ such that $M(G)$ holds and for which $\Delta=4$
	necessarily respect property $P_1$. Indeed, let $u$ be a vertex of maximum
	degree in $G$. Consider $G' = G - N[u]$. If $G'$ contains a vertex of degree
	4, $P_1(G)$ is satisfied. Otherwise, we must have $\Delta(G') = 3$. If $G'$
	is not 3-regular, it is at most 2 cop-win (this holds for orders 12 to 14 as
	mentioned above) and therefore $M(G)$ does not hold. Therefore, any vertex
	in $G'$ that was adjacent to a vertex of $N(u)$ is also of degree 4 in $G$
	and not adjacent to $u$. Hence $P_1(G)$ holds. We can therefore suppose
	$\Delta(G) = 5$ for the remainder of the proof. Furthermore, since $P_1(G)$
	is false, we can assume that if there exists two vertices of maximum degree,
	they must be adjacent (hence they must form a clique).

	We now define property $P_2$. We say a graph $G$ has property $P_2$ if $G$
	contains two non-adjacent vertices $v_1$ and $v_2$ such that $v_1$ has
	degree either 3 or 4, $v_2$ has degree $5$, $v_1$ and $v_2$ have a common
	neighbour, and $G-N[v_1]$ has maximum degree at most $4$. Then, we compute
	the graphs $G$ such that $M(G)$ holds and $P_2(G)$ holds, but not $P_1(G)$.
	Precisely, run the Merging Algorithm by setting $n=18$, $D_2=\Delta=5$,
	$D_1$ to respectively either $3$ or $4$, $\mathcal L_1$ to be the 3-cop-win
	graphs on $12$ vertices with maximum degree at most $4$ (if $u$ is a vertex
	of degree $5$ and $G-N[u]$ has maximum degree 5, then $G$ necessarily
	respect property $P_1$), and $\mathcal L_2$ to be the 3-cop-win graphs on
	respectively either $14$ or $13$ vertices with maximum degree at most $4$.
	We have computed these lists $\mathcal L_1$ and $\mathcal L_2$ in Lemma
	\ref{computationn1213}. The results of this computation are presented in
	Table \ref{tab:algoresultspart2}.

	We note that as the number of possible vertices of maximum degree is
	generally smaller than before, there are fewer graphs thrown out because
	of some vertex $u$ of maximum degree such that $G-N[u]$ is not a 3-cop-win graph.
	Furthermore, we note that as the graphs on 14 vertices with maximum degree 3
	are 3-regular, when choosing any of these graphs as $G_1$ it is impossible
	for $d_1$ to be anything other than $3$.

	\begin{table}[h]
		\centering
		\begin{tabular}{|c|c|c|c|c|c|c|c|c|c|c|}
			\cline{8-11}
			\multicolumn{7}{c}{} & \multicolumn{4}{|c|}{Cop number}
			\\
			\hline
			\rowcolor{gray!60}
			$D_1$                & $G_2$                            & $\Delta_1$
			                     & $G_1$ & $d_1$                            &
			                     Base graphs        & Final graphs & $1$     &
			                     $2$ & $3$                              & $\geq
			                     4$
			                     \\ \hline
			\multirow{4}{*}{4}   & \multirow{4}{*}{1105}            &
			\multirow{3}{*}{4} & \multirow{3}{*}{78}  & $3$
			& $993$              & $41872$      & $0$ & $9299$
			& $32573$            & $0$                               \\
			\cline{5-11} & &                                  &
			& $2$          & $504$   & $70224$              & $0$
			& $4278$             & $65946$      & $0$ \\ \cline{5-11} &
			&                      & & $1$                              & $1138$
			& $3350712$    & $0$     & $417144$ & $2933568$
			& $0$                                                    \\
			\cline{3-11} & & 3                                & $2$
			& $3$          & $153$   & $41006$              & $0$
			& $15440$            & $25566$      & $0$ \\ \hline
			\multirow{2}{*}{3}   & \multirow{2}{*}{$16523$}         &
			\multirow{2}{*}{4} & \multirow{2}{*}{78}  & $2$
			& $2419$             & $83509$      & $0$ & $4187$
			& $79322$            & $0$                               \\
			\cline{5-11} & &                                  &
			& $1$          & $10582$ & $6293171$            & $0$
			& $786173$           & $5506998$    & $0$ \\ \hline
		\end{tabular}
		\caption{Results of the second wave of computations with the Merging Algorithm. It presents the counts for the graphs $G$ built with the property that $G-N[v_1]$ has maximum degree 4, $v_1$ and $v_2$ always have a common neighbour (in particular $d_1<D_1$) and the vertices of maximum degree form a clique. Here, we always have $n=18$ and $D_2=\Delta=5$, and $G_1$ is chosen with maximum degree $\Delta_1$.}
		\label{tab:algoresultspart2}
	\end{table}

	We see that none of the graphs are 4-cop-win. We claim that for all graphs
	$G$, if $M(G)$ holds then either $P_1(G)$ holds or $P_2(G)$ holds. Supposing
	the contrary, let $G$ be as graph such that $M(G)$ holds, but for which
	neither $P_1(G)$ nor $P_2(G)$ holds. We attempt to find a contradiction. As
	discussed earlier, we may only consider the case where $\Delta=5$ as $P_1$
	is always respected for $\Delta=4$.

	Let $v_2$ be a vertex of maximum degree, and denote $G_1=G-N[v_2]$. As $G_1$
	is 3-cop-win (in order for $M(G)$ to hold), $G_1$ contains at most 2
	vertices of degree 1. Otherwise, removing these vertices of degree 1, which
	does not change the cop number, would yield a graph with at most 9 vertices
	and having cop number at most 2 (by Theorem \ref{min3}). By our supposition, we
	know there cannot be a vertex $v_1\in N[v_2]^c$ of degree either 3 or 4 such
	that $v_1$ and $v_2$ have a common neighbour and $G-N[v_1]$ has maximum
	degree at most 4. A sufficient condition for this last requirement is that
	$v_1$ has a neighbour $a$ which is adjacent (or equal) to all vertices of
	maximum degree.

	We first consider the case that $G$ contains a unique vertex of degree 5.
	For any choice of $v_1$ in $N[v_2]^c$, if $v_1$ and $v_2$ have a common
	neighbour, $v_1$ trivially has a common neighbour with all vertices of
	degree 5. So, there are no vertices of $N[v_2]^c$ of degree 3 or 4. Then, no
	vertex of $G_1$-degree $2$ or $3$ has a neighbour in $N(v_2)$, and no vertex
	of $G_1$-degree 1 has more than 1 neighbour in $N(v_2)$. Therefore, only
	vertices of $G_1$-degree 1 can have a neighbour in $N(v_2)$, and even then
	they can only receive 1 each, implying there are at most 2 edges between $N(v_2)$
	and $N[v_2]^c$. On the other hand, as $G$ does not contain a corner, each
	vertex in $N(v_2)$ must have at least 1 neighbour in $N[v_2]^c$. Thus, there
	are at least 5 edges between $N(v_2)$ and the vertices of $N[v_2]^c$. Hence,
	there is a contradiction in this case.

	The other main case is if $N(v_2)$ contains at least one other vertex of
	degree 5. We recall that these vertices of degree 5 must form a clique as
	$G$ does not respect property $P_1$. Hence, if there is a vertex $v_1$ in
	$N[v_2]^c$ which has a neighbour $x\in N(v_2)$ of degree 5 (necessarily in
	$N(v_2)$) and such that $v_1$ has degree 3 or 4, then we have a
	contradiction as $x$ is adjacent (or equal) to all vertices of degree 5.
	Hence, no vertex of degree 5 in $N(v_2)$ can be adjacent to a vertex of
	$G_1$-degree 2 or 3, and its neighbours (of $G_1$-degree 1) can have no
	other neighbours in $N(v_2)$. As $G$ does not contain a corner, each vertex
	in $N(v_2)$ must have at least 1 neighbour in $N[v_2]^c$. Hence, as there are
	at most 2 vertices of $G_1$-degree 1, there can be only be at most 2
	vertices of degree 5 in $N(v_2)$.

	We first consider the case with 2 vertices $x_1,x_2\in N(v_2)$ of degree 5.
	Therefore, $x_1$ and $x_2$ each have exactly one neighbour in $N[v_2]^c$ (which
	are different). This means that there are 2 vertices of $G_1$-degree 1, one of
	which is adjacent to $x_1$ and the other is adjacent to $x_2$. Hence both
	$x_1$ and $x_2$, having degree 5, must have 3 neighbours in $N(v_2)$
	(possibly including each other). In particular, $x_1$ and $x_2$ must have at
	least one common neighbour in $N(v_2)$, as $N(v_2)$ has 5 vertices. Denote
	this common neighbour $y$. We know that $y$ must have a neighbour in
	$N[v_2]^c$ (otherwise $y$ is cornered by $v_2$), which will be our choice of
	$v_1$. As both vertices of $G_1$-degree 1 cannot have more edges with
	vertices in $N(v_2)$, $v_1$ necessarily has degree either $3$ or $4$. As $y$
	is a common neighbour to $v_2$, $x_1$ and $x_2$, this choice of $v_1$ yields
	a contradiction.

	We can similarly consider the case with exactly 1 vertex $x$ of degree 5 in
	$N(v_2)$. Then, $x$ has at most 2 neighbours (of $G_1$-degree 1) in
	$N[v_2]^c$. Hence, to have degree 5, $x$ has at least 2 neighbours $y_1,y_2$
	in $N(v_2)$. We have already seen that they cannot be adjacent to the
	neighbours of $x$ in $N[v_2]^c$. If one of them (say $y_1$) is adjacent to a
	vertex of $G_1$-degree 2 or 3, let $v_1$ be this last vertex. Then, $v_1$ is
	a vertex of degree 3 or 4, and its neighbour $y_1$ is adjacent to both
	vertices of maximum degree ($v_2$ and $x$), which is a contradiction.
	Otherwise, $y_1$ and $y_2$ are both adjacent to some vertex of $G_1$-degree
	1 (to which $x$ was necessarily not adjacent). Choose $v_1$ to be this last
	vertex of $G_1$-degree 1. It has degree at least 3 (one neighbour in $G_1$
	as well as $y_1,y_2$), and its neighbour $y_1$ is adjacent to both vertices
	of maximum degree ($v_2$ and $x$). This is a contradiction.

	Thus, our claim is verified: all graphs $G$ that satisfy $M$ respect either
	$P_1$ or $P_2$. We have computed the cop number of graphs such that $M(G)$
	and $P_1(G)$ hold, and graphs such that $M(G)$ and $P_2(G)$ hold but not
	$P_1(G)$. Hence, we have computed the cop number of all graphs such that
	$M(G)$ holds. This proves the current proposition.
\end{proof}

It is interesting that in the first part of computations, for all cases with
$\Delta=4$, not only is the Robertson graph the only 4-cop-win graph, but there
are no other candidate 4-cop-win graphs. A reasonable explanation for this might
be that when merging, too many vertices of high degree are created: either a
vertex of degree 5 or more is created (in which case the graph is immediately
thrown out) or there are "too many" vertices of degree 4, such that there is
always some $u$ of maximum degree for which $G-N[u]$ not 3-cop-win.

\subsection{Possible improvements}

This is only one of many possible computational approaches to solving the
problem. We now discuss a few improvements and alternatives that the interested
reader may want to apply.

Our approach was based on merging 3-cop-win graphs by looking at non-adjacent
vertices $v_1,v_2$. It is easy to see that one could also choose $v_1$ and $v_2$
to be adjacent. Even if the construction would be somewhat different, the ideas
are similar. In particular, after proving that $G$ does not contain non-adjacent
vertices of maximum degree, we could have proved that $G$ does not contain any
adjacent vertices of maximum degree, instead of considering $v_1$ of smaller
degree. This would then leave only the case with a single vertex of maximum
degree to be treated. With some additional heuristics or with a simplification
of the methods we used, this case could be dealt with more specifically.

We decided against this approach for few reasons. Although our approach required
us to compute more 3-cop-win graphs than otherwise needed, it allowed us to
implement only one Merging Algorithm. Furthermore, computing the 3-cop-win
graphs on 14 vertices with maximum degree (at most) 4 allowed us to
simultaneously handle on the case on $n=18$ with $d(v_1)=3$, and build the
candidate 4-cop-win graphs on 19 vertices with maximum degree 4.

Another method would be to not only merge graphs relative to pairs of vertices,
but varying sizes of subsets. This approach would certainly reduce the number of
intermediate graphs generated by the algorithm: instead of pruning graphs
after adding edges, we could build up a larger part of the graph. The difficulty
lies in implementing this approach. In particular, the structure of the
partially-constructed graphs would not be as simple, as there will be much more
than 6 or 7 types of vertices.

Although at the expense of some computation time, we have chosen not to
implement these improvements in order to keep the code as simple as possible.
Indeed, the simplicity of the code reduces the chances of it being erroneous, as
well as making it easier to verify. As the proof is completely dependent on the
results of the algorithm, we felt this compromise was justified.

A last idea essentially combines the processes of generating the graphs and
testing their cop number. Let $G$ be a connected graph and $e$ be some edge of
$G$. In general, it is unclear whether removing $e$ will help the robber or help
the cops, as this depends on many other factors. If we consider a slightly
modified ruleset so that the robber can use the edge $e$ but not the cops, we
might achieve some results. Denote $c'$ the cop number of this modified game.
With these rules, $c(G)\leq c'(G)$, as the new restrictions on $e$ can only
benefit the robber. Furthermore, $c'(G-e)\leq c'(G)$ because removing $e$ can
only help the cops, as they were not allowed to use it anyway. Thus, both
$c(G)$ and $c(G-e)$ are bounded above by $c'(G)$. If we modify the algorithm
which calculates the cop number to take into account the robber-only edge $e$ (a
fairly easy modification), we could then determine simultaneously whether both
$G$ and $G-e$ have cop number at most 3. This generalizes to larger subsets of
edges. Hence, in theory, we can reduce by a significant amount the number of
cases to consider by not distinguishing $G$ and $G-e$. If done for many edges,
this idea could decrease the number of graphs to consider exponentially. It is
not clear how many such "special edges" we can take in $G$ before the cop
number with the special edges diverges from the cop number without those edges.
We leave implementing and studying this approach as a problem. Note that modifying slightly the rules of
the game to study the cop number has been done many times before. For instance,
cop-only edges are an important tool in \cite{frankl_pursuit_1987} and allowing
the cops to teleport is used in \cite{lehner_cop_2021}.



\section{Main results}\label{mainsection} We are now ready to prove the desired
results.

\begin{theorem}\label{mainthm1} If $G$ is a connected graph such that $n\leq
		18$, then $c(G)\leq 3$.
\end{theorem}
\begin{proof}
	The statement for all cases except $n=17$ with $\Delta=4$ and $n=18$ with
	$\Delta=4,5$ follows from Corollary \ref{maxdegnlarger11} and Propositions
	\ref{propn11} and \ref{propdegree3}. Then, apply Proposition \ref{propremainingcases}
	first for $n=17$ then for $n=18$. Indeed, applying our results in this order
	will allow us to know that every proper induced connected subgraph $H$ of
	$G$ respects $c(H)\leq 3$ as $H$ necessarily has smaller order.

	See Table \ref{tab:applythms} for a visual guide to which proposition to
	apply in each case.
\end{proof}

\begin{table}[h]
	\centering
	\begin{tabular}{|c|cccccc|}
		\cline{2-7}
		\multicolumn{1}{c}{} & \multicolumn{6}{|c|}{$n$}
		\\
		\hline
		$\Delta$     & \cellcolor{gray!60}\hspace{0.17cm} $\leq
		14$\hspace{0.18cm} & \cellcolor{gray!60}15                    &
		\cellcolor{gray!60}16                    & \cellcolor{gray!60}17
		& \cellcolor{gray!60}18                      & \cellcolor{gray!60}19
		\\
		\hline
		\cellcolor{gray!60}$\leq 3$ & \cellcolor{green!70}\ref{propdegree3}
		& \cellcolor{green!70}\ref{propdegree3}    &
		\cellcolor{green!70}\ref{propdegree3}    &
		\cellcolor{green!70}\ref{propdegree3}      &
		\cellcolor{green!70}\ref{propdegree3}      &
		\cellcolor{green!70}\ref{propdegree3}             \\
		\cellcolor{gray!60}4        & \cellcolor{blue!60}\ref{maxdegnlarger11}
		& \cellcolor{yellow}\ref{propn11} (1)      &
		\cellcolor{orange!70}\ref{propn11} (2)   &
		\cellcolor{red!70}\ref{propremainingcases} &
		\cellcolor{red!70}\ref{propremainingcases} &
		\cellcolor{red!70}\ref{propremainingcases}        \\
		\cellcolor{gray!60}5        & \cellcolor{blue!60}\ref{maxdegnlarger11}
		& \cellcolor{blue!60}\ref{maxdegnlarger11} &
		\cellcolor{yellow}\ref{propn11} (1)      &
		\cellcolor{orange!70}\ref{propn11} (2)     &
		\cellcolor{red!70}\ref{propremainingcases} &
		\\
		\cellcolor{gray!60}6        & \cellcolor{blue!60}\ref{maxdegnlarger11}
		& \cellcolor{blue!60}\ref{maxdegnlarger11} &
		\cellcolor{blue!60}\ref{maxdegnlarger11} &
		\cellcolor{yellow}\ref{propn11} (1)        &
		\cellcolor{orange!70}\ref{propn11} (2)     &
		\\
		\cellcolor{gray!60}7        & \cellcolor{blue!60}\ref{maxdegnlarger11}
		& \cellcolor{blue!60}\ref{maxdegnlarger11} &
		\cellcolor{blue!60}\ref{maxdegnlarger11} &
		\cellcolor{blue!60}\ref{maxdegnlarger11}   &
		\cellcolor{yellow}\ref{propn11} (1)        &
		\cellcolor{magenta!50}\ref{propositionn19D78} (2) \\
		\cellcolor{gray!60}8        & \cellcolor{blue!60}\ref{maxdegnlarger11}
		& \cellcolor{blue!60}\ref{maxdegnlarger11} &
		\cellcolor{blue!60}\ref{maxdegnlarger11} &
		\cellcolor{blue!60}\ref{maxdegnlarger11}   &
		\cellcolor{blue!60}\ref{maxdegnlarger11}   &
		\cellcolor{magenta!75}\ref{propositionn19D78} (1) \\
		\cellcolor{gray!60}$\geq 9$ & \cellcolor{blue!60}\ref{maxdegnlarger11}
		& \cellcolor{blue!60}\ref{maxdegnlarger11} &
		\cellcolor{blue!60}\ref{maxdegnlarger11} &
		\cellcolor{blue!60}\ref{maxdegnlarger11}   &
		\cellcolor{blue!60}\ref{maxdegnlarger11}   &
		\cellcolor{blue!60}\ref{maxdegnlarger11}          \\
		\hline
	\end{tabular}
	\caption{Proposition/corollary to apply for each case in Theorems
	\ref{mainthm1} and \ref{mainthm2}. For Propositions \ref{propn11} and
	\ref{propositionn19D78} we specify the case using the numbering from the
	proof.}
	\label{tab:applythms}
\end{table}

Considering there exists a known 4-cop-win graph on 19 vertices, the Robertson
graph, we get the following corollary.
\begin{corollary}
	$M_4=19$.
\end{corollary}

We also want to narrow down the possible 4-cop-win graphs on 19 vertices.

\begin{theorem}\label{mainthm2} Let $G$ be a connected graph such that $n=19$.
	If $\Delta\leq 3$ or $\Delta \geq 7$, then $c(G)\leq 3$. If $\Delta=4$, then
	$c(G)\leq 3$, unless $G$ is the Robertson graph.
\end{theorem}
\begin{proof}
	This is a direct consequence of Corollary \ref{maxdegnlarger11} and
	Propositions \ref{propositionn19D78}, \ref{propdegree3} and
	\ref{propremainingcases} (in the last case using Theorem \ref{mainthm1} to
	say that $c(H)\leq 3$ when $H$ is a proper induced connected subgraph of
	$G$).

	See Table \ref{tab:applythms} for a visual guide to which proposition to
	apply in each case.
\end{proof}

We leave filling the missing cases in this theorem as a conjecture.

\begin{conjecture}
	There does not exist a connected graph $G$ such that $n=19$, $\Delta\in
	\{5,6\}$ and $c(G)=4$.
\end{conjecture}

This would show that the Robertson graph is the unique 4-cop-win graph on 19
vertices. With a better implementation of the algorithm, in some low overhead
programming language such as C, and with a few additional good ideas, this
problem seems within reach. On the other hand, finding $M_5$ with the methods
used in this article is out of reach. A fully non-computational proof
$M_4=19$ would also be of great interest, and may be more instructive on how to
approach $M_5$.

It is asked in \cite{baird_minimum_2014} whether the minimum $d$-cop-win graphs
are $(d,5)$-cage graphs for every $d$. Although we now have further evidence
pointing towards this conjecture, any general proof of this statement is still
beyond our grasp.

\section*{Acknowledgments}
We thank Seyyed Aliasghar Hosseini for introducing us to this and other similar
problems and for many helpful discussions. We are grateful to Ben Seamone and
Ge\v na Hahn for their support over the course of this project. We also thank
the reviewers for their comments, which have improved the presentation and
clarity of this article.

We also acknowledge that most computations were done on the Université de
Montréal - Département de mathématiques et de statistique computer network.

We thank the Natural Sciences and Engineering Research Council of Canada (NSERC), the Fonds de Recherche du Qu\'ebec - Nature et technologies (FRQNT) and McGill University for their financial support. Nous remercions le Conseil de recherches en sciences naturelles et en génie du Canada (CRSNG), le Fonds de Recherche du Qu\'ebec - Nature et technologies (FRQNT) et l'Université McGill pour leur soutien financier.

\bibliography{newrefs}
\bibliographystyle{abbrv}

\end{document}